\documentclass[10pt]{amsart}

\usepackage{geometry,graphicx,amssymb,amsmath,amsbsy,eucal,amsfonts,mathrsfs,amscd,bm}
\usepackage{color}
\usepackage[all]{xy}
\usepackage{tikz-cd}
\usetikzlibrary{arrows,shapes,calc,snakes}
\usetikzlibrary{shapes,snakes}
\usepackage{tikz-3dplot}
\usetikzlibrary{intersections}
\usepackage{tkz-euclide}
\usetikzlibrary{angles}
\usepackage{pgfplots}
\usepackage{eufrak,enumitem}

\numberwithin{equation}{section}

\allowdisplaybreaks[3]

\newtheorem{theorem}{Theorem}[section]
\newtheorem{lemma}[theorem]{Lemma}

\newtheorem{remark}[theorem]{Remark}
\newtheorem{corollary}[theorem]{Corollary}

\theoremstyle{definition}
\newtheorem{definition}[theorem]{Definition}

\newcommand{\mcta}{\mct^\text{\rm a}}
\newcommand{\mctwf}{\mct^\text{\rm wf}}
\newcommand{\Th}{T_h}

\newcommand{\sss}{\ensuremath\mathop{\mathrm{s}}}
\newcommand{\rr}{\ensuremath\mathop{\mathrm{r}}}

\newcommand{\p}{{\partial}}

\newcommand{\mct}{\mathcal{T}_h}

\newcommand{\jump}[1]{\left[\hspace{-0.025in}\left[#1\right]\hspace{-0.025in}\right]}

\newcommand{\curl}{{\ensuremath\mathop{\mathrm{curl}\,}}}
\newcommand{\dive}{{\ensuremath\mathop{\mathrm{div}\,}}}

\newcommand{\jmp}[1]{ [\![ {#1}  ]\!] }

\newcommand{\pol}{\mathcal{P}}

\newcommand{\bbR}{\mathbb{R}}

\newcommand{\VV}{\mathcal{V}}

\newcommand{\WFT}{T^{\ensuremath\mathop{\mathrm{wf}\,}}}
\newcommand{\THWFT}{\mathcal{T}_h^{\ensuremath\mathop{\mathrm{wf}\,}}}
\newcommand{\FCT}{F^{\ensuremath\mathop{\mathrm{ct}\,}}}

\newcommand{\grad}{{\ensuremath\mathop{\mathrm{grad}\,}}}

\newcommand{\TA}{T^{\ensuremath\mathop{\mathrm{a}\,}}}
\newcommand{\RTPi}{\Pi^{\ensuremath\mathop{\mathrm{RT}\,}}}

\newcommand{\WFK}{K^{\ensuremath\mathop{\mathrm{wf}\,}}}

 \newcommand{\mV}{\mathring{V}}
\newcommand{\LL}{\mathsf{L}}
\newcommand{\Fct}{F^{\text{\rm ct}}}
\newcommand{\Ta}{T^{\text{\rm a}}}
\newcommand{\Twf}{T^{\text{\rm wf}}}
\newcommand{\Kwf}{K^{\text{\rm wf}}}
\newcommand{\divFi}{{\rm div}_{F_i}\,}
\newcommand{\divF}{{\rm div}_{F}\,}

\newcommand{\AL}[1]{{\color{black}#1}}
\newcommand{\MJN}[1]{{\color{black}#1}}

\newcommand{\revj}[1]{{\color{black}#1}}
\newcommand{\revjtwo}[1]{{\color{black}#1}}

\newcommand{\tang}{\upsilon}

\begin{document}
\title{Exact sequences on  Worsey-Farin Splits}


%

\author[Johnny Guzm\'an]{Johnny Guzm\'an}
\address{Division of Applied Mathematics, Brown University, Providence, RI 02912 }
\email{johnny\_guzman@brown.edu}           
\thanks{The first author was supported in part by NSF grant DMS--1913083.  The third
author was supported in part by NSF grant DMS--2011733.}

\author[Anna  Lischke]{Anna Lischke}   
\address{Division of Applied Mathematics, Brown University, Providence, RI 02912 }
\email{anna\_lischke@brown.edu} 

\author[Michael Neilan]{Michael Neilan}   
\address{Department of Mathematics, University of Pittsburgh, Pittsburgh, PA 15260}
\email{neilan@pitt.edu}

\maketitle

\begin{abstract}
We construct several smooth finite element spaces
defined on three--dimensional Worsey--Farin splits.
In particular, we construct $C^1$, $H^1(\curl)$, and $H^1$-conforming
finite element spaces and show  the discrete spaces satisfy
local exactness properties.  A feature of the spaces is their low polynomial degree
and lack of extrinsic supersmoothness at sub-simplices of the mesh.
In the lowest order case, the last two spaces in the sequence consist of piecewise linear and piecewise
constant spaces, and are suitable for the discretization
of the (Navier-)Stokes equation.
\end{abstract}

\thispagestyle{empty}
\section{Introduction}

An inherent feature of smooth finite element spaces, with respect
to a general simplicial mesh, is their high polynomial degree and complexity.
For example, $C^1$-conforming finite element spaces necessitate the use
of polynomials of at least degree five and nine in two and three dimensions, respectively \cite{ciarlet2002finite,lai2007spline}.
Another feature of smooth piecewise polynomial spaces is their complexity, as additional
smoothness is imposed on lower-dimensional simplices of the mesh.
 For example, in three dimensions, $C^1$ piecewise polynomials are $C^4$ on vertices and $C^2$ on edges of the mesh \cite{lai2007spline,Zenisek73,Zhang09A}.
 
 Recently, the connection between $C^1$ finite element spaces and stable divergence--free (Stokes) pairs
 for incompressible flows has been emphasized through the use of smooth, discrete de Rham complexes 
 (cf., e.g., \MJN{\cite{FalkNeilan2013,FuGuzman,GuzmanNeilan14R,JohnEtal2017}}).  The relationships between distinct finite element spaces
 imply  many of the attributes of smooth finite element spaces (high polynomial degree and complexity)
 translate to divergence--free pairs.

One way to mitigate the high polynomial degree and complexity of smooth finite element spaces, and analogously 
divergence--free Stokes pairs,
is to define the spaces on certain splits (or refinements) of a simplicial triangulation; the added structure
of the split mesh offers additional flexibility not available on generic meshes.
For example, an Alfeld split of a simplex connects each vertex to its barycenter, thus splitting each $n$-simplex
into $(n+1)$ sub-simplices; this is commonly referred to as a Clough--Tocher split in two-dimensions \cite{ciarlet2002finite}.
The polynomial degree of $C^1$ spaces on Alfeld splits is dramatically 
reduced from five to three in two dimensions, and from nine to five in three dimensions.
These $C^1$ spaces are related to the divergence--free Scott-Vogelius pair for the (Navier-)Stokes problem, where
the velocity space consists of continuous piecewise polynomials and the pressure space consists of discontinuous
polynomials of one degree less \cite{ScottVogelius85}.  On Alfeld splits, the Scott-Vogelius pair is stable if the polynomial degree of the velocity
space is at least the \revj{spatial} dimension \cite{arnold1992quadratic,FuGuzman,guzman2018inf,zhang2005}.

While reducing the polynomial degree, these finite element spaces defined on Alfeld splits still have
supersmoothness at low-dimensional simplices, e.g., the three-dimensional $C^1$ elements on Alfeld splits 
are $C^2$ at vertices \AL{\cite{FuGuzman}}. Moreover, there is still a restriction of polynomial degree for the corresponding Scott--Vogelius pair, which is especially limiting in three dimensions.  These issues motivate the use of other types of splits with more 
facets, in particular, the three-dimensional Worsey--Farin split \cite{WORSEY1988,lai2007spline}.  Similar to the Alfeld split,
the Worsey-Farin split adds a vertex to the interior of each tetrahedron and connects this vertex
to its (four) vertices.  In addition, the Worsey-Farin split adds a vertex to each face of the tetrahedron
and connects this vertex to the vertices of the face and to the interior vertex.    Thus, a 2D Alfeld split is performed
on each face of the tetrahedron and the split produces $12$ sub-tetrahedra (cf.~Section \ref{sec:Prelims} for the precise
construction and definitions).

The goal of this paper is to construct finite element spaces with varying level of smoothness
defined on Worsey--Farin splits in three dimensions.  We connect several local finite element
spaces defined on these splits through the use of a discrete de Rham complex and show
that the sequences are exact for any polynomial degree.  The exactness properties
naturally lead to dimension formulas for the local piecewise polynomial spaces.
\MJN{These dimension formulas appear to be new, even 
for the $C^1$ spaces, and of independent interest.}
We then construct unisolvent sets of degrees of freedom for the spaces which lead 
to the analogous global spaces and commuting projections.   The last two spaces
in the sequences are suitable for the discretization of the (Navier)-Stokes problem.

Features of the proposed finite element spaces are their low polynomial degree and lack
of extrinsic supersmoothness.   The lowest-order $C^1$ finite element space consists
of piecewise cubic polynomials \MJN{with respect to the Worsey-Farin triangulation}, and the accompanying Scott-Vogelius (Stokes) pair consists
of spaces of piecewise linear and  piecewise constant spaces for the velocity and pressure, respectively. 
\MJN{We emphasize that, compared to the analogous spaces defined on Alfeld splits, the polynomial degree is reduced
by two.}
\MJN{In addition}, the degrees of freedom
of the \MJN{proposed} spaces only use derivative information dictated by their global smoothness, and therefore
the spaces do not have added continuity restrictions on lower-dimensional simplices in the mesh.
\MJN{Again, this contrasts with the finite element spaces defined on Alfeld splits.}

One of the characterizations of a Worsey--Farin split is the presence of singular edges, i.e., 
edges that fall in exactly two planes; this is analogous to \MJN{two-dimensional singular vertices, i.e.,
vertices  falling
on exactly two straight lines in a (two-dimensional) triangulation mesh.
It is well known that in two dimensions, the divergence of piecewise smooth vector fields
have a weak continuity property at such points, and this intrinsic smoothness 
characterizes the discrete pressure spaces in a Stokes/NSE finite element discretization \cite{Vogelius83,ScottVogelius85}.
Analogously, we show that}
the derivatives of (continuous) piecewise polynomials
have intrinsic smoothness properties on singular edges.  \MJN{For example}, the divergence operator
acting on the Lagrange finite element space has a alternating weak continuity property
on singular vertices, and this affects the last space in the sequence (the ``pressure space''). 
\MJN{Similar results, but in less generality, are shown in \cite[Lemma 3.1]{ZhangS2011} and \cite[Section 6]{Neilan15MC}. 
For the first time, we also characterize intrinsic smoothness properties of the curl operator
acting on the Lagrange finite element space at singular edges (cf.~Remark \ref{rem:thetaprop}).
}

This paper is a continuation and \MJN{nontrivial} extension of \cite{GuzmanLischkeNeilan}, where
smooth piecewise polynomial spaces are built on  two-dimensional Powell--Sabin meshes.
The present work also has similarities with the recent work by Christiansen
and Hu \cite{christiansen2018generalized}, where low--order finite element de Rham complex are constructed
on several different meshes (splits).  However, unlike this work, 
we build all of our finite element spaces on a Worsey--Farin split and for general polynomial degree.
\MJN{One of the main differences in the construction and the analysis
between the current work and those given in \cite{christiansen2018generalized,FuGuzman,GuzmanLischkeNeilan}
is identifying weak continuity properties for both the divergence and curl operator
at at singular edges.  This necessitates the construction of spaces (both locally and globally)
with additional smoothness of the  N\'ed\'elec spaces, yet are not globally continuous; cf. \eqref{eqn:3dVV2def}--\eqref{eqn:3dVV3odef}.}
Stokes pairs defined on Worsey--Farin splits have also been analyzed in \cite{ZhangS2011} using 
the quadratic Lagrange space for the velocity.  However, the pressure space in \cite{ZhangS2011}
was not explicitly characterized as we do in the current paper.

The rest of the paper is organized as follows.  In the next section,
we provide the notation and definitions used throughout the paper.
In Section \ref{sec:LocalSeq} we show that  local smooth finite element
spaces satisfy exactness properties with respect to several de Rham complex.
This is proved, in part, by using the exactness properties of piecewise polynomials
defined on two-dimensional Clough-Tocher splits.  These exactness properties
naturally lead to dimension formulas for the local spaces, which we state in Section 
\ref{sec-dim}. 
Next, we give unisolvent sets of degrees of freedom (DOFs) for each space
in Section \ref{sec-DOFs} and show that the DOFs induce commuting projections.
Finally, in Section \ref{sec-Global}, we prove that the DOFs lead
to global (conforming) finite element spaces.

This paper is based on the second author's Ph.D. thesis \cite{ALphd}.

\section{Preliminaries}\label{sec:Prelims}
Let $\Omega \subset \mathbb{R}^3$ be a contractible polyhedral domain.  We assume we have a shape-regular simplicial triangulation $\mct$ of $\Omega$. For each $T \in \mct$ we let $z_T$  denote its incenter\MJN{, that is,} the center of the largest inscribed ball contained in $T$. Let $F=\overline{T_1} \cap \overline{T_2}$ be an interior face with $T_1, T_2 \in \mct$. Let $L$ be the line segment connecting $z_{T_1}$ and $z_{T_2}$; then we let $m_F=L \cap F$. Since we chose $z_{T}$ to be the incenters, we can guarantee \AL{that} $m_F$ exists. If $F$  is a boundary face of $\mct$, then we let $m_F$ be the barycenter of $F$. \revj{For a simplex $K$,   
$\Delta_s(K)$ will denote the \MJN{set of} $s$-subsimplices \MJN{(i.e., the $s$-dimensional subsimplices)} of $K$. More generally, if $\mathcal{F}_h$ is a collection of simplices, then $\Delta_s(\mathcal{F}_h)$ denotes the collection of $s$-subsimplices of all the simplices in $\mathcal{F}_h$. Moreover, if $\mathcal{F}_h$ is a simplicial triangulation of a domain with boundary, then $\Delta_s^I(\mathcal{F}_h)$ denotes the collection of $s$-subsimplices of $\Delta_s(\mathcal{F}_h)$ that do not belong to the boundary of the domain.} 

For each $T \in \mct$ with $T=[x_0, \ldots, x_3]$, we let $\Ta=\{K_i, 0 \le i \le 3\}$ \MJN{with} $K_i=[z_T, x_0, \ldots, \widehat{x_i}, \ldots, x_3]$.  \revj{ Here and throughout $\widehat{\cdot}$ represents omission of the term}. In other words,  we see that $\Ta$ is a triangulation of $T$ with four simplices and this is known as the Alfeld split of $T$.  Let $F_i= [x_0, \ldots, \widehat{x_i}, \ldots, x_3]$ be the $i$-th face of $T$ \MJN{so that} $K_i   \in \Ta$ with $F_i \subset K_i$. Then we let $\Kwf_i=\{ S_j, 0 \le j \le 3, j \neq i \}$ where $S_j=[z_T, m_{F_i},   x_{k}, x_{\ell}]$ and $0 \le k, \ell \le 3$ with  $k, \ell \notin \{i, j\}$.  We let $\Twf=\{ S \in \Kwf_i:  0 \le i \le 3\}$; cf.~Figure \ref{fig:WF}.   We see that $\Twf$ consists of $12$ simplices and is known as the Worsey-Farin split of $T$. We let  $\mcta:=\{ K \in \Ta: T \in \mct\}$ and $\mctwf:=\{ K \in \Twf: T \in \mct\}$. We see that $\mcta$ is a refinement of $\mct$ and $\mctwf$ is a refinement of $\mcta$.  \revj{For $T \in \mct$, $\mu$ denotes the hat function corresponding to $z_T$ defined  on the Alfeld split  $\Ta$ . That is, $\mu$ is the piecewise linear function with respect to $\Ta$ such that $\mu(z_T)=1$ and $\mu=0$ on $\partial T$. Moreover, we use the notation $\mu_i=\mu|_{K_i}$}.

For any $F \in \Delta_2(T)$ for $T \in \mct$ we see that $\Twf$ induces a Clough-Tocher triangulation of $F$, which we denote by $\Fct$. To be precise, let $F=[y_0, y_1, y_2]$, then $\Fct:=\{ [m_F, y_0, y_1],  [m_F, y_0, y_2],  [m_F, y_1, y_2] \}$. We will utilize surface differential operators. Let $\psi$ be a smooth enough vector valued function on $T \in \mct$  and let $F \in \Delta_2(T)$. Then we let the tangential part of $\psi$  be given by $\psi_F:=n \times \psi \times n |_F$ \revj{ where $n$ is the outward pointing normal of $\partial T$}. For a scalar valued function $u$,  we define $u_F:= u|_F$. We will use the following identities
\begin{alignat}{2}
{\rm curl}_F\, \psi_F  =& \curl \psi \cdot n \qquad && \text{ on } F,  \label{curlid}\\
 {\rm div}_F\,  \psi_F =& \dive (n \times \psi \times n)  \qquad && \text{ on } F,  \label{divid} \\
 {\rm grad}_F\, u_F =& n \times \grad u \times n  \qquad && \text{ on } F \label{gradid},\\
\MJN{ {\rm rot}_F\,u_F =} & \MJN{\grad u \times n \qquad }&&\MJN{\text{on } F.}
\end{alignat}

\begin{figure}\centering
\includegraphics[width=.5\textwidth]{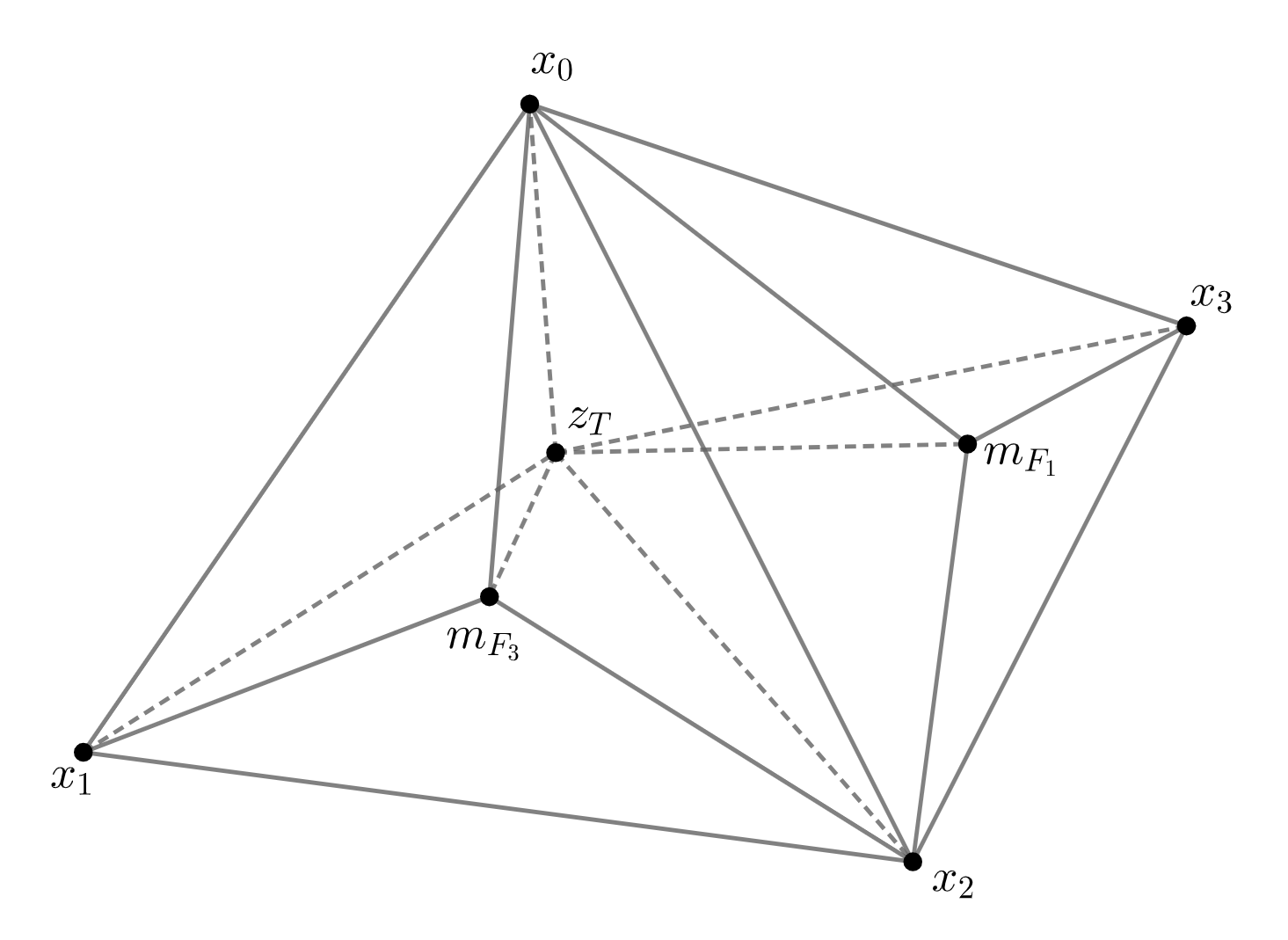}
\caption{\label{fig:WF}A representation of the Worsey-Farin split with two faces shown.}
\end{figure}

We now define local finite element spaces on each macro-tetrahedron $T \in \mct$. To do this, we assume we have a triangulation of $T$, $\Th$, which could be, for example, $\Th=\Ta$ or $\Th=\Twf$. We let $\pol_r(S)$ be the space of polynomials of degree  less than or equal to $r$ defined on $S$. \MJN{For negative values of $r$, $\pol_r(S)$ is the trivial set.} We define on the triangulation $\Th$ of $T$ the space of discontinuous polynomials:
\begin{equation*}
\pol_r(\Th):=\{ v \in L^2(T): v|_K \in \pol_r(K), \forall K \in \Th\}. 
\end{equation*}

The spaces of minimal smoothness are defined as follows.
\begin{alignat*}{2}
V_r^0(\Th):=& \pol_r(\Th) \cap H^1(T), \quad && \mV_{\revj{r}}^0(\Th):=\mathring{H}^1(T) \cap V_{\revj{r}}^0(\Th), \\
V_r^1(\Th):=&  [\pol_r(\Th)]^3 \cap  H({\rm curl}; T), \quad && \mV_{\revj{r}}^1(\Th):=\mathring{H}({\rm curl},T) \cap V_{\revj{r}}^1(\Th), \\
V_r^2(\Th):=&  [\pol_r(\Th)]^3 \cap  H({\rm div}; T), \quad && \mV_{\revj{r}}^2(\Th):=\mathring{H}({\rm div},T) \cap V_{\revj{r}}^2(\Th), \\
V_r^3(\Th):=&\pol_r(\Th) , \quad && \mV_{\revj{r}}^2(\Th):=L_0^2(T) \cap V_{\revj{r}}^2(\Th),
\end{alignat*}
where $L_0^2(T) = \{q \in L^2(T) : \int_T q\MJN{\,dx} = 0\}$. \revj{Here,  we use the commonly used notation that $\mathring{(\cdot)}$ denotes the corresponding space with vanishing traces}\MJN{, e.g., 
$\mathring{H}({\rm curl},T) = \{v\in H({\rm curl},T):\ v\times n|_{\p T} = 0\}$
and $\mathring{H}({\rm div},T) = \{v\in H({\rm div},T):\ v\cdot n|_{\p T} = 0\}$.}
We also consider the Lagrange finite elements $\LL_r^0(\Th):= V_r^0(\Th)$ \big($ \mathring{\LL}_r^0(\Th):= \mV^0(\Th)$\big) , $\LL_r^i(\Th):=[\LL_r^0(\Th)]^3$  \big($ \mathring{\LL}_r^{\revj{i}}(\Th):=[\mathring{\LL}_r^0(\Th)]^3$\big)  for $i=1,2$ and finally $\LL_r^3(\Th):=V_r^0(\Th)$  \big($ \mathring{\LL}_r^3(\Th):= L_0^2(T) \cap \revj{\mathring{\LL}}_r^0(\Th)$\big). Finally, we define the smoother finite element \MJN{spaces}
\begin{alignat*}{2}
S_r^0(\Th):=&\{ v \in \LL_r^0(\Th) : \grad v \in \LL_{r-1}^{1}(\Th) \},  \quad && \mathring{S}_r^0(\Th):=\{ v \in \mathring{\LL}_r^0(\Th) : \grad v \in \mathring{\LL}_{r-1}^{1}(\Th) \}, \\
S_r^1(\Th):=&\{ v \in \LL_r^1(\Th) : \curl v \in \LL_{r-1}^{2}(\Th) \},  \quad && \mathring{S}_r^1(\Th):=\{ v \in \mathring{\LL}_r^1(\Th) : \curl v \in \mathring{\LL}_{r-1}^{2}(\Th) \}, \\
S_r^2(\Th):=&\{ v \in \LL_r^2(\Th) : \dive v \in \LL_{r-1}^{3}(\Th) \},  \quad && \mathring{S}_r^2(\Th):=\{ v \in \mathring{\LL}_r^2(\Th) : \dive  v \in \mathring{\LL}_{r-1}^{3}(\Th) \}, \\
S_r^3(\Th):=&\LL_r^3(\Th),   \quad && \mathring{S}_r^3(\Th):=\mathring{\LL}_r^3(\Th).
\end{alignat*}

We also define the \MJN{intermediate}  spaces that add extra smoothness to the spaces $V_r^i(\Twf)$ on the faces of $T$.
\begin{subequations}
\begin{align}
	\VV_r^2(\Twf) &= \{v \in V_r^2(\Twf) : v \times n \text{ is continuous on each } F \in \Delta_2(T) \}, \label{eqn:3dVV2def}\\
	\mathring{\VV}_r^2(\Twf) &= \{ v \in \VV_r^2(\Twf) : v \cdot n = 0 \text{ on each } F \in \Delta_2(T)\}, \label{eqn:3dVV2odef}\\
	\VV_r^3(\Twf) &= \{q \in V_r^3(\Twf) : q \text{ is continuous on each } F \in \Delta_2(T)\}, \label{eqn:3dVV3def}\\
	\mathring{\VV}_r^3(\Twf) &= \VV_r^3(\Twf) \cap L_0^2(T). \label{eqn:3dVV3odef}
\end{align}
\end{subequations}
We see that  $\LL_r^i(\Twf) \subset \VV_r^i(\Twf) \subset V_r^i(\Twf) $ for $i=2,3$.

\section{Local Exact Sequences}\label{sec:LocalSeq}
\MJN{One of the main results of the paper} is to prove local sequences \MJN{consisting
of smooth piecewise polynomials} are exact. The first sequences are the ones with homogeneous boundary conditions.
\begin{subequations}\label{WFbdryseqs}
\begin{alignat}{5}
&0\
{\xrightarrow{\hspace*{0.5cm}}}\
\mathring{V}_r^0(\Twf)\
&&\stackrel{\grad}{\xrightarrow{\hspace*{0.5cm}}}\
\mathring{V}_{r-1}^1(\Twf)\
&&\stackrel{\curl}{\xrightarrow{\hspace*{0.5cm}}}\
\mathring{V}_{r-2}^2(\Twf)\
&&\stackrel{\dive}{\xrightarrow{\hspace*{0.5cm}}}\
\mathring{{V}}_{r-3}^3(\Twf) 
&&\xrightarrow{\hspace*{0.5cm}}\
 0,\label{WFbdryseq1}\\
&0\
{\xrightarrow{\hspace*{0.5cm}}}\
\mathring{S}_r^0(\Twf)\
&&\stackrel{\grad}{\xrightarrow{\hspace*{0.5cm}}}\
\mathring{{\LL}}_{r-1}^1(\Twf)\
&&\stackrel{\curl}{\xrightarrow{\hspace*{0.5cm}}}\
\mathring{{\VV}}_{r-2}^2(\Twf)\
&&\stackrel{\dive}{\xrightarrow{\hspace*{0.5cm}}}\
\mathring{{V}}_{r-3}^3(\Twf)\
&&\xrightarrow{\hspace*{0.5cm}}\
 0,\label{WFbdryseq2}\\
 &0\
{\xrightarrow{\hspace*{0.5cm}}}\
\mathring{S}_{r}^0(\Twf)\
&&\stackrel{\grad}{\xrightarrow{\hspace*{0.5cm}}}\
\mathring{S}_{r-1}^1(\Twf)\
&&\stackrel{\curl}{\xrightarrow{\hspace*{0.5cm}}}\
\mathring{\LL}_{r-2}^2(\Twf)\
&&\stackrel{\dive}{\xrightarrow{\hspace*{0.5cm}}}\
\mathring{\VV}_{r-3}^3(\Twf)\
&&\xrightarrow{\hspace*{0.5cm}}\
 0,\label{WFbdryseq3}\\
  &0\
{\xrightarrow{\hspace*{0.5cm}}}\
\mathring{S}_{r}^0(\Twf)\
&&\stackrel{\grad}{\xrightarrow{\hspace*{0.5cm}}}\
\mathring{S}_{r-1}^1(\Twf)\
&&\stackrel{\curl}{\xrightarrow{\hspace*{0.5cm}}}\
\mathring{S}_{r-2}^2(\Twf)\
&&\stackrel{\dive}{\xrightarrow{\hspace*{0.5cm}}}\
\mathring{\LL}_{r-3}^3(\Twf)\
&&\xrightarrow{\hspace*{0.5cm}}\
 0. \label{WFbdryseq4}
\end{alignat}
\end{subequations}

The second set of sequences does not have boundary conditions. 
 \begin{subequations}\label{WFseqs}
 \begin{alignat}{5}
&\mathbb{R}\
\rightarrow\
{{V}}_{r}^0(\Twf)\
&&\stackrel{\grad}{\xrightarrow{\hspace*{0.5cm}}}\
{{V}}_{r-1}^1(\Twf)\
&&\stackrel{\curl}{\xrightarrow{\hspace*{0.5cm}}}\
{{V}}_{r-2}^2(\Twf)\
&&\stackrel{\dive}{\xrightarrow{\hspace*{0.5cm}}}\
{{V}}_{r-3}^3(\Twf)\
%
&&\rightarrow \
 0,\label{seq1}\\
 &\mathbb{R}\
\rightarrow \
{S}_{r}^0(\Twf)\
&&\stackrel{\grad}{\xrightarrow{\hspace*{0.5cm}}}\
{\LL}_{r-1}^1(\Twf)\
&&\stackrel{\curl}{\xrightarrow{\hspace*{0.5cm}}}\
{V}_{r-2}^2(\Twf)\
&&\stackrel{\dive}{\xrightarrow{\hspace*{0.5cm}}}\
{V}_{r-3}^3(\Twf)\
%
&&\rightarrow \
 0,\label{seq2}\\
  &\mathbb{R}\
\rightarrow\
{S}_{r}^0(\Twf)\
&&\stackrel{\grad}{\xrightarrow{\hspace*{0.5cm}}}\
{S}_{r-1}^1(\Twf)\
&&\stackrel{\curl}{\xrightarrow{\hspace*{0.5cm}}}\
{\LL}_{r-2}^2(\Twf)\
&&\stackrel{\dive}{\xrightarrow{\hspace*{0.5cm}}}\
{V}_{r-3}^3(\Twf)\
%
&&\rightarrow \
 0, \label{seq3}\\
  &\mathbb{R}\
\rightarrow\
{S}_{r}^0(\Twf)\
&&\stackrel{\grad}{\xrightarrow{\hspace*{0.5cm}}}\
{S}_{r-1}^1(\Twf)\
&&\stackrel{\curl}{\xrightarrow{\hspace*{0.5cm}}}\
{S}_{r-2}^2(\Twf)\
&&\stackrel{\dive}{\xrightarrow{\hspace*{0.5cm}}}\
{\LL}_{r-3}^3(\Twf)\
%
&& \rightarrow \
 0. \label{seq4}
 \end{alignat}
 \end{subequations}
 
The first sequences \eqref{WFbdryseq1} and \eqref{seq1} are exact \MJN{due to the results of} N\'ed\'elec \cite{Nedelec86}. The major result of this section are the following theorems. 
\begin{theorem}\label{bdryseqs}
Let $r \geq 3$. Then the sequences \eqref{WFbdryseqs} are exact. 
\end{theorem}

\begin{proof}
Again, we already know that \eqref{WFbdryseq1} is exact. The exactness of the rest of the sequences follow from the following results that are found below: Theorem \ref{divthm},  Theorem \ref{curlthm}, and Theorem \ref{gradthm}.
\end{proof}

Similarly, we can prove the following theorem. 
\begin{theorem}\label{seqs}
Let $r \geq 3$. Then the sequences \eqref{WFseqs} are exact. 
\end{theorem}

\MJN{The proofs of Theorems \ref{bdryseqs}--\ref{seqs}
follow the same general procedure as \cite{guzman2018inf,FuGuzman,GuzmanLischkeNeilan}.
Essentially, we build functions of the form $\mu w_1+\mu^2 w_2+\cdots$
through an iterative procedure,
where the functions $w_j$ are specified
such that the results
of Theorems \ref{bdryseqs}--\ref{seqs} are inferred.  However,
due to geometric properties of the Worsey-Farin split,
there are non-trivial differences between the arguments given in
\cite{guzman2018inf,FuGuzman,GuzmanLischkeNeilan} and those presented here.
Essentially this is due to the induced Clough-Tocher
triangulation on each face $F\subset \p T$, which from
to the definition of $\mu$,
 implies that the functions $w_j$ 
 are piecewise polynomials with respect to the Worsey-Farin split;
 in contrast, in \cite{guzman2018inf,FuGuzman}, the functions $w_j$
 are simply polynomials on $T$ defined by canonical N\'ed\'elec degrees of freedom.
As such, the exactness
of polynomial sequences defined on (two-dimensional) Clough-Tocher triangulations
plays an {essential} role in the proofs of  Theorems \ref{bdryseqs}--\ref{seqs}.
}

\subsection{Exact sequences on Clough--Tocher splits}
Before we prove Theorems \ref{bdryseqs} and  \ref{seqs}, 
we will need to use exact sequence properties of local Clough-Tocher splits. 
\MJN{To start, we require some definitions.
\begin{definition}\label{def:nst}

For a tetrahedron $T\in \mathcal{T}_h$ and face $F\in \Delta_2(T)$,
we denote by $n_F:=n|_F$ the outward unit normal of $\p T$ restricted to $F$.
We let $\tau$ and $\tang$ be orthonormal 
vectors that span the tangent space of $F$.
Thus, $\{\tau,\tang,n_F\}$ is an orthonormal system of $\bbR^3$.
%
%
\end{definition}
}

\revj{
\begin{remark} If $P(F)$ is a scalar valued space defined on $F$ then, as commonly done, and with an abuse of notation,  we set $[P(F)]^2=\{ a t+ b \tang: a, b \in P(F)\}$.
\end{remark}
}

We define the  \revj{N\'ed\'elec} spaces on the Clough--Tocher split:
\begin{alignat*}{3}
&V_{{\rm div}, r}^1(\Fct):= \{ v \in H({\rm div}_F, F): v|_Q \in [\pol_r(F)]^2, \forall Q \in \Fct \},\quad && \mathring{V}^1_{{\rm div},r}(\Fct) = V_{{\rm div}, r}^1(\Fct)\cap 
\mathring{H}({\rm div}_F, F)\\
&V_{{\rm curl}, r}^1(\Fct):= \{ v \in H({\rm curl}_F;,F): v|_Q \in [\pol_r(F)]^2, \forall Q \in \Fct \},\quad && \mathring{V}^1_{{\rm curl},r}(\Fct) = V_{\rm curl, r}^1(\Fct)\cap 
\mathring{H}({\rm curl}_F, F),\\
&{V}_r^2(\Fct):=\{ v \in L^2(F): v|_Q \in \pol_r(F), \forall Q \in \Fct\},\quad &&\mathring{V}_r^2(\Fct):=V_r^2(\Fct)\cap  L^2_0(F),
\end{alignat*}
and the Lagrange spaces,
\begin{alignat*}{3}
&\LL_r^0(\Fct) := V_r^2(\Fct)\cap H({\rm grad}_F;F),\quad && \mathring{\LL}_r^0(\Fct) := \LL_r^0(\Fct)\cap \mathring{H}({\rm grad}_F;F),\\
&\LL_r^1(\Fct) := [\LL_r^0(\Fct)]^2,\quad && \mathring{\LL}_r^1(\Fct) := [\mathring{\LL}_r^0(\Fct)]^2,\\
&\LL_r^2(\Fct) := \LL_r^0(\Fct),\quad && \mathring{\LL}_r^2(\Fct) := \mathring{\LL}_r^0(\Fct)\cap L^2_0(F).
\end{alignat*}
Finally, we define the subspaces with additional smoothness.
\begin{alignat*}{2}
&S_r^0(\Fct) :=\{v\in \LL_r^0(\Fct):\ {\rm grad}_F\,v \in \revj{\LL_{r-1}^1(\Fct)} \},\quad &&\mathring{S}_r^0(\Fct) :=\{v\in \mathring{\LL}_r^0(\Fct):\ {\rm grad}_F\,v \in \revj{\mathring{\LL}_{r-1}^1(\Fct)} \},\\
&S_{{\rm div},r}^1(\Fct)  :=\{v\in \LL_r^1(\Fct):\ {\rm div}_F\,v\in \revj{\LL_{r-1}^2(\Fct)} \},\quad &&\mathring{S}_{{\rm div},r}^1(\Fct) :=\{v\in \mathring{\LL}_r^1(\Fct):\ {\rm div}_F\,v\in \revj{\mathring{\LL}_{r-1}^2(\Fct)} \},\\
&S_{{\rm curl},r}^1(\Fct) :=\{v\in \LL_r^1(\Fct):\ {\rm curl}_F\,v\in \revj{\LL_{r-1}^2(\Fct)} \},\quad &&\mathring{S}_{{\rm curl},r}^1(\Fct) :=\{v\in \mathring{\LL}_r^1(\Fct):\   {\rm curl}_F\,v \in \revj{\mathring{\LL}_{r-1}^2(\Fct)}  \},\\
&S_r^2(\Fct) :=\LL_r^2(\Fct),\quad && \mathring{S}_r^2(\Fct) :=\mathring{\LL}_r^2(\Fct).
\end{alignat*}
\revj{We note that $V_{{\rm curl}, r}^1(\Fct) \big(\MJN{resp.,} \, S_{{\rm curl}, r}^1(\Fct) \big)$ and $V_{{\rm div}, r}^1(\Fct) \big(\MJN{resp.,} \, S_{{\rm div}, r}^1(\Fct) \big)$ \MJN{are isomorphic}. For notational convenience we sometimes drop the ${\rm curl}$ and ${\rm div}$ from the subscripts} \MJN{of these spaces}.

Several combinations of these spaces form exact sequences, which are summarized below.


\begin{theorem}\label{2dseqs}
Let $r \geq 1$. The following sequences are exact \cite{arnold1992quadratic,FuGuzman}.
\begin{subequations}
\begin{alignat}{4}
&\mathbb{R}\
{\xrightarrow{\hspace*{0.5cm}}}\
{{\LL}}_{r}^0(\FCT)\
&&\stackrel{{\rm grad}_F}{\xrightarrow{\hspace*{0.5cm}}}\
{{V}}_{{\rm curl},r-1}^1(\FCT)\
&&\stackrel{{\rm curl}_F}{\xrightarrow{\hspace*{0.5cm}}}\
{{V}}_{r-2}^2(\FCT)\
&&\xrightarrow{\hspace*{0.5cm}}\
 0,\label{alfseq1}\\
 &\mathbb{R}\
{\xrightarrow{\hspace*{0.5cm}}}\
{S}_{r}^0(\FCT)\
&&\stackrel{{\rm grad}_F}{\xrightarrow{\hspace*{0.5cm}}}\
{\LL}_{r-1}^1(\FCT)\
&&\stackrel{{\rm curl}_F}{\xrightarrow{\hspace*{0.5cm}}}\
{V}_{r-2}^2(\FCT)\
&&\xrightarrow{\hspace*{0.5cm}}\
 0,\label{alfseq2}\\
  &\mathbb{R}\
{\xrightarrow{\hspace*{0.5cm}}}\
{S}_{r}^0(\FCT)\
&&\stackrel{{\rm grad}_F}{\xrightarrow{\hspace*{0.5cm}}}\
{S}_{\revj{\rm curl}, r-1}^1(\FCT)\
&&\stackrel{{\rm curl}_F}{\xrightarrow{\hspace*{0.5cm}}}\
{\LL}_{r-2}^2(\FCT)\
&&\xrightarrow{\hspace*{0.5cm}}\
 0, \label{alfseq3} \\
&0\
{\xrightarrow{\hspace*{0.5cm}}}\
{\mathring{\LL}}_{r}^0(\FCT)\
&&\stackrel{{\rm grad}_F}{\xrightarrow{\hspace*{0.5cm}}}\
{\mathring{V}}_{\curl,r-1}^1(\FCT)\
&&\stackrel{{\rm curl}_F}{\xrightarrow{\hspace*{0.5cm}}}\
{\mathring{V}}_{r-2}^2(\FCT)\
&&\xrightarrow{\hspace*{0.5cm}}\
 0,\label{2dbdryseq1}\\
 &0\
{\xrightarrow{\hspace*{0.5cm}}}\
\mathring{S}_{r}^0(\FCT)\
&&\stackrel{{\rm grad}_F}{\xrightarrow{\hspace*{0.5cm}}}\
\mathring{\LL}_{r-1}^1(\FCT)\
&&\stackrel{{\rm curl}_F}{\xrightarrow{\hspace*{0.5cm}}}\
\mathring{V}_{r-2}^2(\FCT)\
&&\xrightarrow{\hspace*{0.5cm}}\
 0,\label{2dbdryseq2}\\
  &0\
{\xrightarrow{\hspace*{0.5cm}}}\
\mathring{S}_{r}^0(\FCT)\
&&\stackrel{{\rm grad}_F}{\xrightarrow{\hspace*{0.5cm}}}\
\mathring{S}_{\revj{\rm curl}, r-1}^1(\FCT)\
&&\stackrel{{\rm curl}_F}{\xrightarrow{\hspace*{0.5cm}}}\
\mathring{\LL}_{r-2}^2(\FCT)\
&&\xrightarrow{\hspace*{0.5cm}}\
 0. \label{2dbdryseq3}
 \end{alignat}
 \end{subequations}
 \end{theorem}
Theorems \ref{2dseqs} has an alternate form that follows from a rotation of the coordinate axes, where the operators ${\rm grad}$ and ${\rm curl}$ are replaced by ${\rm rot}$ and ${\rm div}$, respectively.

\begin{corollary}\label{cor:rotdiv}
Let $r \geq 1$. The following sequences are exact \cite{arnold1992quadratic,FuGuzman}.
\begin{subequations}
\begin{alignat}{4}
&\mathbb{R}\
{\xrightarrow{\hspace*{0.5cm}}}\
{{\LL}}_{r}^0(\FCT)\
&&\stackrel{{\rm rot}_F}{\xrightarrow{\hspace*{0.5cm}}}\
{{V}}_{{\rm div},r-1}^1(\FCT)\
&&\stackrel{{\rm div}_F}{\xrightarrow{\hspace*{0.5cm}}}\
{{V}}_{r-2}^2(\FCT)\
&&\xrightarrow{\hspace*{0.5cm}}\
 0,\label{altalfseq1}\\
 &\mathbb{R}\
{\xrightarrow{\hspace*{0.5cm}}}\
{S}_{r}^0(\FCT)\
&&\stackrel{{\rm rot}_F}{\xrightarrow{\hspace*{0.5cm}}}\
{\LL}_{r-1}^1(\FCT)\
&&\stackrel{{\rm div}_F}{\xrightarrow{\hspace*{0.5cm}}}\
{V}_{r-2}^2(\FCT)\
&&\xrightarrow{\hspace*{0.5cm}}\
 0,\label{altalfseq2}\\
  &\mathbb{R}\
{\xrightarrow{\hspace*{0.5cm}}}\
{S}_{r}^0(\FCT)\
&&\stackrel{{\rm rot}_F}{\xrightarrow{\hspace*{0.5cm}}}\
{S}_{\revj{\rm div}, r-1}^1(\FCT)\
&&\stackrel{{\rm div}_F}{\xrightarrow{\hspace*{0.5cm}}}\
{\LL}_{r-2}^2(\FCT)\
&&\xrightarrow{\hspace*{0.5cm}}\
 0, \label{altalfseq3} \\
&0\
{\xrightarrow{\hspace*{0.5cm}}}\
{\mathring{\LL}}_{r}^0(\FCT)\
&&\stackrel{{\rm rot}_F}{\xrightarrow{\hspace*{0.5cm}}}\
{\mathring{V}}_{\dive,r-1}^1(\FCT)\
&&\stackrel{{\rm div}_F}{\xrightarrow{\hspace*{0.5cm}}}\
{\mathring{V}}_{r-2}^2(\FCT)\
&&\xrightarrow{\hspace*{0.5cm}}\
 0,\label{alt2dbdryseq1}\\
 &0\
{\xrightarrow{\hspace*{0.5cm}}}\
\mathring{S}_{r}^0(\FCT)\
&&\stackrel{{\rm rot}_F}{\xrightarrow{\hspace*{0.5cm}}}\
\mathring{\LL}_{r-1}^1(\FCT)\
&&\stackrel{{\rm div}_F}{\xrightarrow{\hspace*{0.5cm}}}\
\mathring{V}_{r-2}^2(\FCT)\
&&\xrightarrow{\hspace*{0.5cm}}\
 0,\label{alt2dbdryseq2}\\
  &0\
{\xrightarrow{\hspace*{0.5cm}}}\
\mathring{S}_{r}^0(\FCT)\
&&\stackrel{{\rm rot}_F}{\xrightarrow{\hspace*{0.5cm}}}\
\mathring{S}_{\revj{\rm div}, r-1}^1(\FCT)\
&&\stackrel{{\rm div}_F}{\xrightarrow{\hspace*{0.5cm}}}\
\mathring{\LL}_{r-2}^2(\FCT)\
&&\xrightarrow{\hspace*{0.5cm}}\
 0. \label{alt2dbdryseq3}
 \end{alignat}
 \end{subequations}
\end{corollary}

\begin{remark}
From Theorem \ref{2dseqs}, Corollary \ref{cor:rotdiv}, the well-known dimension formulas
of the \MJN{N\'ed\'elec} and Lagrange spaces,
and the rank--nullity theorem, one easily finds
the dimensions of the smoother spaces $S^k_r(\Fct)$ (cf.~\cite{FuGuzman}).  These counts
are summarized in Table \ref{tab:2DDim}.

\end{remark}

\begin{table}
\caption{\label{tab:2DDim}Dimension counts of the canonical (two--dimensional) N\'ed\'elec, Lagrange, and smooth spaces with and without boundary conditions
with respect to the Clough--Tocher split.  Here, $\dim V_{{\rm div},r}^1(\Fct) = \dim V_{{\rm curl},r}^1(\Fct) =:\dim V_r^1(\Fct)$}
{\scriptsize 
\begin{tabular}{c|ccc}
& $k=0$ & $k=1$ & $k=2$\\
\hline
$\dim V_r^k(\Fct)$ & $\frac12(3r^2 + 3r + 2)$  & $3(r+1)^2$ & $\frac{3}{2}(r+1)(r+2)$\\
$\dim \mathring{V}_r^k(\Fct)$ & $ \frac12(3r^2 - 3r + 2)$ & $3r(r+1)$ & $\frac{3}{2}(r+1)(r+2) - 1$\\
$\dim \LL_r^k(\Fct)$ & $ \frac{1}{2}(3r^2 + 3r + 2)$ & \MJN{$3r^2 + 3r + 2$} & $\frac{1}{2}(3r^2 + 3r + 2)$\\
$\dim \mathring{\LL}_r^k(\Fct)$ & $\frac{1}{2}(3r^2 - 3r + 2)$ & $3r^2 - 3r + 2$ & $\frac{3}{2}r(r-1)$\\
$\dim S_r^k(\Fct)$ & $ \frac{3}{2}(r^2 - r +2)$ & $3r^2 + 3$ & $\frac{1}{2}(3r^2 + 3r + 2)$\\
$\dim \mathring{S}_r^k(\Fct)$ & $ \frac{3}{2}(r^2 - 5r + 6)$ & $ 3r^2 -9r+6$ & $\frac{3}{2}r(r - 1)$\\
$\dim \mathcal{R}_{r}^k(\Fct)$ & $\frac{3}{2}(r-1)(r-2)$ & $3(r-1)^2$ &---
\end{tabular}
}
\end{table}

We will use the following intermediate spaces when developing commuting projections on the Worsey-Farin split:
\[
\mathcal{R}_r^0(\FCT) := \{v \in S_r^0(\FCT) : v|_{\partial F} = 0\},\quad \mathcal{R}_r^1(\FCT):= \{v \in S_{\dive, r}^1(\FCT) : v|_{\partial F} = 0\}. 
\]
It is shown in \cite{ALphd} that $\dim \mathcal{R}_r^0(\FCT) = \frac{3}{2}(r-1)(r-2)$ and $\dim \mathcal{R}_r^1(\FCT) = 3(r-1)^2.$

\subsection{Surjectivity of the divergence operator on discrete local spaces}\label{sec-sur-local-div}
The goal of this section is to prove the following theorem. 

\begin{theorem}\label{divthm}
Let $r \ge 0$. Then:
\begin{enumerate}[align=left, leftmargin=0pt, labelindent=\parindent]
\item[(i)] for each $p \in \mathring{\VV}_r^3(\Twf)$, there exists a $v \in \mathring{\LL}_{r+1}^2(\Twf)$ such that $\dive v = p$.\smallskip
\item[(ii)] for each $p \in \mathring{V}_r^3(\Twf)$, there exists a $v \in \LL_{r+1}^2(\Twf) \cap \mathring{V}_{r+1}^2(\Twf)$ such that $\dive v = p$.\smallskip
\item[(iii)] for each $p \in V_r^3(\Twf)$, there exists a $v \in \LL_{r+1}^2(\Twf)$ such that $\dive v = p$. \smallskip
\item[(iv)] for each $p \in \mathring{\LL}_r^3(\Twf)$ (resp., $p \in \LL_r^3(\Twf)$), there exists a $v \in \mathring{S}_{r+1}^2(\Twf)$ (resp., $v \in S_{r+1}^2(\Twf)$) such that $\dive v = p$.
\end{enumerate}
\end{theorem}

The proofs of Theorem \ref{divthm} parts (i) and (ii) depends on five preliminary lemmas. 
\begin{lemma}\label{lemma1}
{Let $r\ge 1$ and $s\ge 0$ be integers.  Then for any $q \in \VV_r^3(\Twf)$, there exists} $w \in \LL_r^2(\Twf)$ and $g \in V_{r-1}^3 (\Twf)$,  such that $\mu^s q = \dive (\mu^{s+1}w) + \mu^{s+1}g$.
\end{lemma}
\begin{proof} 
{Let $q\in \VV_r^3(\Twf)$ and $s\ge 0$.
Because $q|_{F_i}$ is continuous on each $F_i\in \Delta_2(T)$, 
there exists $b_i\in \pol_r(F_i)$ such that $b_i = q|_{F_i}$ on $\p F_i$.}
Thus $q-b_i$ is continuous on $F_i$ and vanishes on $\partial F_i$. 
Consequently, there exists $a_i\in \LL_r^{\revj 3}(\Twf)$ such that $a_i = (q-b_i)$ on $F_i$ and $\text{supp}(a_i)\subseteq K_i$.
Using the divergence-conforming N\'ed\'elec degrees of freedom of the second kind \cite{Nedelec86} , and the fact that $\grad \mu_i$ is parallel
to the outward unit normal of $F_i$, there exists $w_1 \in [\pol_r(T)]^3$ such that
\begin{align*}
	(s+1)w_1 \cdot \grad \mu_i = b_i \quad \text{on } F_i.
\end{align*}
 We also define $w_2 \in \LL_r^2(\Twf)$ \MJN{as}
\begin{align*}
	w_2 := \frac{1}{s+1}\sum_{i=1}^4 a_i \ell_i,
\end{align*}
where  $\ell_i: = \frac{\grad \mu_i}{|\grad \mu_i|^2}$. Finally, we set $w: = w_1+w_2 \in \LL_r^2(\Twf)$.  We then see that 
\begin{equation*}
 q- (s+1)w \cdot \grad \mu=0 \quad \text{ on } \partial T,
\end{equation*}
and, hence, there exists $p \in  V_{r-1}^3(\Twf)$ such that 
\begin{equation*}
 q= (s+1)w \cdot \grad \mu+ \mu p  \quad  \text{ on } T.
\end{equation*}
Setting $g := \revj{p -\dive w } \in V_{r-1}^3(\Twf)$ we have 
\begin{alignat*}{1}
\mu^s q=&   (s+1) \mu^s w \cdot \grad \mu+ \mu^{s+1} p 
= \dive( \mu^{\revj {s+1}} w)+ \mu^{s+1} g.
\end{alignat*}
\end{proof}

\begin{lemma}\label{lemmaaux861}
\revj{ Let $K \in \Ta$, $F \in \Delta_2(T)$ with $F \subset \partial K$ and let $n_F$ denote the outward pointing unit normal to $F$}. If $p \in \mathring{\LL}_r^1(\Fct)$ then there exists $q \in  \LL_r^1(\Twf)$ such that {$q|_F = p$}, $\text{supp}(q) \subset K$ and $q \cdot \revj{n_F}=0$ on $K$. 
\end{lemma}
\begin{proof}
Let  \MJN{$\{\tau, \tang, n_F\}$} be an orthonormal set with \MJN{$\tau$ and $\tang$} parallel to $F$. Then \MJN{$p=a\tau+b\tang$} for some 
$a,b \in \mathring{\LL}_r^0(\Fct)$. We  extend $a$ and $b$ to all of $K$, which we denote by $\tilde{a}, \tilde{b} \in \revj{\LL}_r^0(\Kwf)$, by setting all the other Lagrange degrees of freedom to be zero. In particular $\tilde{a}$ and $\tilde{b}$ vanish on $\partial K \backslash F$. Hence, we can further extend them by zero to all of $T$ to obtain  $\tilde{a}, \tilde{b} \in \revj{\LL}_r^0(\Twf)$. We then set \MJN{$q=\tilde{a} \tau+ \tilde{b} \tang$}.
\end{proof}

\begin{lemma}\label{lemma2}
For any $\theta \in V_r^3(\Twf)$, with $r\ge 0$, there exists $\psi \in \LL_{r+1}^2(\Twf) \cap \mathring{V}_{r+1}^2(\Twf)$ and $\gamma \in \VV_r^3(\Twf)$ such that
\begin{align}
\mu^s \theta &= \dive (\mu^s \psi) + \mu^s \gamma\qquad \forall s\ge 0.
\end{align}
\end{lemma}
\begin{proof}  
Let $K_i \in \revj{\Ta}$ be the tetrahedron containing the face $F_i \in \Delta_2(T)$, and let \MJN{$\kappa_i \in V_0^3(\Ta) \subset V_0^3(\Twf)$} be defined on $K_i$ as $\kappa_i = \frac{1}{|F_i|}\int_{F_i} \theta\MJN{\,dA}$. Then on $F_i$,
\begin{align*}
\int_{F_i}(\theta - \kappa_i)\MJN{\,dA} = 0,
\end{align*}
so $(\theta-\kappa_i)_{F_i} \in \mathring{V}_r^2(\Fct_i)$ by definition. Hence, by Corollary \ref{cor:rotdiv}, there exists a function $\rho_i \in \mathring{\LL}_{r+1}^1(\Fct_i)$ such that
\begin{alignat}{1}
	\divFi \rho_i  = (\theta-\kappa_i) \quad \text{ on }{F_i}. \label{thetakappa}
\end{alignat}

\revj{By Lemma \ref{lemmaaux861}}, there exists an extension $\psi_i \in \LL_{r+1}^{{2}}(\Twf)$ 
such that $\psi_i\revj{|_{F_i}}=\rho_i$,  $\text{supp}(\psi_i) \subseteq K_i$, and $\psi_i \cdot n_{F_i}=0 $ on $K_i$. We then define $\psi = \sum_{i=0}^3 \psi_i \in  \LL_{r+1}^{{2}}(\Twf)  \cap \mathring{V}_{r+1}^2(\Twf)$. The construction of $\psi$, and using \eqref{divid},  
yields the identities 
\begin{alignat}{2}
\psi \cdot n_{F_i} =&  0 \qquad &&\text{on } K_i, \label{eqn:psinormal}\\
\dive \psi=& \divFi\rho_i\qquad &&{\text{on }F_i}.  \label{eqn:dive}
\end{alignat}
Now  set $\gamma:= \theta-\dive \psi$, so that $\gamma=\kappa_i$ on $F_i$ by \eqref{eqn:dive} and \eqref{thetakappa}. Since $\kappa_i$ is continuous on $F_i$, it follows that $\gamma \in \VV_r^3(\Twf)$. Rearranging \MJN{terms} yields $\theta = \dive \psi + \gamma$, which proves the result in the case $s = 0$. Furthermore, since $\grad \mu$ is parallel to $n_{F_i}$ on each $K_i$, we have by \eqref{eqn:psinormal},
\begin{align*}
	\mu^s \theta - \dive(\mu^s \psi) &= \mu^s \theta - \mu^s \dive \psi - s \mu^{s-1}\psi \cdot \grad \mu= \mu^s \gamma,
\end{align*}
which is the desired result.
\end{proof}

\begin{lemma}\label{lemma3}
Let $q \in \VV_r^3(\Twf)$ with $r \geq 1$, {and $s\ge 0$}. Then there exists $v \in \LL_r^2(\Twf)$ and $p \in \VV_{r-1}^3(\Twf)$ such that $\mu^s q = \dive (\mu^{s+1}v) + \mu^{s+1} p$. 
\end{lemma}
\begin{proof} 
By Lemma \ref{lemma1}, there exist $w \in \LL_r^2(\Twf)$ and $g \in V_{r-1}^3(\Twf)$ such that
\begin{align*}
\mu^s q &= \dive(\mu^{s+1}w) + \mu^{s+1} g. 
\end{align*}
Since $g \in V_{r-1}^3(\Twf)$, Lemma \ref{lemma2} yields {the existence of }$\psi \in \LL_r^2(\Twf)$ and $p \in \VV_{r-1}^3(\Twf)$ such that
\begin{align*}
\mu^{s+1} g = \dive (\mu^{s+1} \psi) + \mu^{s+1} p.
\end{align*}
Therefore, $\mu^s q = \dive(\mu^{s+1}(w+\psi)) + \mu^{s+1}p$. Setting $v = w+\psi$ achieves the desired result.
\end{proof}

The final preliminary lemma follows from a  result shown in \cite{FuGuzman}. 
\begin{lemma}\label{lemma4}
Let $s \geq 0$, and let $q \in \VV_0^3(\Twf)$ with $\int_T \mu^s q\MJN{\,dx} = 0$. Then there exists $w \in \LL_0^2(\Twf)$ such that $\mu^s q = \dive(\mu^{s+1}w)$.
\end{lemma}
\begin{proof}
 Since $q \in \VV_0^3(\Twf)$ it is easy to see that $q \in V_0^3(\Ta)$. From  \cite[Lemma 3.11]{FuGuzman},
 there exists $w \in [\pol_{0}(T)]^3 \subset \LL_0^2(\Twf)$ such that $\dive (\mu^{s+1} w)= \revj{\mu^s} q$. 
\end{proof}

We can now prove {Theorem \ref{divthm} parts (i) and (ii).}
\begin{proof}[Proof of Theorem \ref{divthm}, part (i)]
\revjtwo{The case $r=0$ follows immediately from Lemma \ref{lemma4} with $s=0$. Now consider the case $r \ge 1$.  Let  $0 \le j \le r-1$ and assume that we have found $w_r, w_{r-1}, \ldots, w_{r-j}$ with $w_{\ell} \in \LL_{\ell}^2(\Twf)$  and $p_{r-j} \in \VV_{r-j}^3(\Twf)$ such that
\begin{alignat*}{1}
p=&\dive(\mu w_r + \mu^2 w_{r-1} + \cdots + \mu^{j+1} w_{r-j})+ \mu^{j+1}p_{r-(j+1)}.
\end{alignat*}
If $0\le j <r-1$ then we apply Lemma \ref{lemma3} to find $w_{r-(j+1)} \in \LL_{r-(j+1)}^2(\Twf)$ and $p_{r-(j+2)} \in {\VV_{r-(j+2)}^3(\Twf)}$ such that \begin{alignat*}{1}
\mu^{j+1}p_{r-(j+1)}=\dive( \mu^{j+2} w_{r-(j+1)})+\mu^{j+2}p_{r-(j+2)}.
\end{alignat*}
In which case we obtain
\begin{alignat*}{1}
p=&\dive(\mu w_r + \mu^2 w_{r-1} + \cdots + \mu^{j+2} w_{r-(j+1)})+ \mu^{j+2}p_{r-(j+2)}.
\end{alignat*}
After taking care of the base case $j=0$, and continuining by induction we arrive at
\begin{align*}
p=\dive (\mu w_r + \mu^2 w_{r-1} + \cdots + \mu^r w_1) +\mu^r p_0.
\end{align*}
By the hypothesis $\int_T p \AL{\,dx} =0$, there holds $\int_T \mu^r p_0\MJN{\,dx} = 0$. By Lemma \ref{lemma4}, there exists $w_0 \in \LL_0^2(\Twf)$ such that $\dive(\mu^{r+1}w_0) = \mu^r p_0$. The result follows by setting $v = \mu w_r + \mu^2 w_{r-1} + \cdots + \mu^r w_1 + \mu^{r+1} w_0$.}

\end{proof}

\begin{proof}[Proof of Theorem \ref{divthm}, part (ii)]
By Lemma \ref{lemma2} (with $s=0$), there exists  $\psi \in \LL_{r+1}^2(\Twf) \cap \mathring{V}_{r+1}^2(\Twf)$ and $\gamma \in \VV_r^3(\Twf)$ satisfying
\begin{equation*}
p= \dive \psi + \gamma.
\end{equation*}
Note that  $\int_{T} p \AL{\,dx} =0$, and \MJN{$\int_{T} \dive \psi\, dx =\int_{\partial T} \psi \cdot n\, dA=0$} since $\psi \cdot n=0$ on $\partial T$. Thus,  we have that $\int_{T} \gamma\MJN{\, dx}=0$ which implies $\gamma \in \mathring{\VV}_r^3(\Twf)$. Therefore, we  apply part (i) of Theorem \ref{divthm} to find $g \in \mathring{\LL}_{r+1}(\Twf)$ such that $\dive g=\gamma$. The  result follows by setting $v=\psi+g$.
\end{proof}

We now prove parts (iii) and (iv) of Theorem \ref{divthm}, which are corollaries to parts (i) and (ii).  
\begin{proof}[Proof of Theorem \ref{divthm}, part (iii)]
We decompose $p=(p-\overline{p})+ \overline{p}$ where $\overline{p}:=\frac{1}{|T|} \int_T p\MJN{\,dx}$. There exists $w \in [\pol_1(T)]^3 $ such that $\dive w= \overline{p}$, and by part (ii) of Theorem \ref{divthm} we have $\psi \in \LL_{r+1}^2(\Twf) \cap \mathring{V}_{r+1}^2(\Twf)$ such that $\dive \psi=p-\overline{p}$. Thus, setting $v:=\psi+w$ completes the proof. 
\end{proof}

\begin{proof}[Proof of Theorem \ref{divthm}, part (iv)] 
Let $p \in {\mathring{\LL}_r^3(\Twf)} \subset \mathring{\VV}_r^3(\Twf)$. Applying part (i) of Theorem \ref{divthm}, we 
find $v \in {\mathring{\LL}_{r+1}^2(\Twf)}$ such that $\dive v = p$. But clearly $v \in \mathring{S}_{r+1}^2(\Twf)$, since $\dive v$ \revj{belongs to $\mathring{\LL}_r^3(\Twf)$}. 
\end{proof}

\subsection{Surjectivity of the curl operator on discrete local spaces}
The main goal of this section is to derive the analagous results
of Section \ref{sec-sur-local-div}, but for the curl operator; that is,
we show \AL{that} the curl operator acting on piecewise polynomial spaces
with respect to the Worsey--Farin split is surjective onto spaces
of divergence--free functions.  The main result of this section is the following. 



\begin{theorem}\label{curlthm}
Let $r \ge 0$. Then:
\begin{enumerate}[align=left, leftmargin=0pt, labelindent=\parindent]
\item[(i)] for any $v \in \mathring{\VV}_r^2(\Twf)$ satisfying $\dive v = 0$ there exists $w \in \mathring{\LL}_{r+1}^1(\Twf)$ satisfying $\curl w = v$.\smallskip
\item[(ii)] let $v \in V_r^2(\Twf)$ with $\dive v = 0$. Then there exists $w \in {\LL_{r+1}^1(\Twf)}$ such that $\curl w = v$.\smallskip
\item[(iii)]  for each $v \in {\mathring{\LL}_r^2(\Twf)}$ (resp., $v \in {\LL_r^2(\Twf)}$) \MJN{with} $\dive v = 0$, there exists a $w \in \mathring{S}_{r+1}^1(\Twf)$ (resp., $w \in S_{r+1}^1(\Twf)$) such that $\curl w = v$. \smallskip
\item[(iv)]  for each $v \in \mathring{S}_r^2(\Twf)$ (resp., $v \in S_r^2(\Twf)$) \MJN{with} $\dive v = 0$, there exists $w \in \mathring{S}_{r+1}^1(\Twf)$ 
\MJN{(resp., $w \in S_{r+1}^1(\Twf)$)} such that $\curl w = v$.
\end{enumerate}
\end{theorem}




We omit the proofs of parts (iii) and (iv) of Theorem \ref{curlthm} 
since they easily follow  from parts (i) and (ii) of the same theorem. 

Before we prove parts (i) and (ii) of Theorem \ref{curlthm}, 
we first establish several lemmas. 
\begin{lemma}\label{curlProp}
Let $r \ge 0$ and let $v \in \mathring{\VV}_r^2(\Twf)$. Then there exist functions $z \in [\pol_r(\Twf)]^3$ and $\gamma \in [\pol_{r-1}(\Twf)]^3$ such that
\begin{align}\label{eqn:curlidentity}
	v = \grad \mu \times z + \mu \gamma,
\end{align}
and so  $ \grad \mu \times z$ is continuous on $F$ for each $F \in \Delta_2(T)$. \AL{Moreover, for each $e \in \Delta_1(T)$ and its unit tangent vector $t$, $z \cdot t$ is single-valued on $e$.}
\end{lemma}
\begin{proof}
By \cite[Lemma 4.1]{FuGuzman}, there exists  
$z \in [\pol_r(\Twf)]^3$ and $\gamma \in [\pol_{r-1}(\Twf)]^3$ such that \eqref{eqn:curlidentity} holds. For each $F \in \Delta_2(T)$,
there holds $v=  \grad \mu \times z $ on $F$, and hence $ \grad \mu \times z$  is continuous on $F$. Following exactly the proof of \cite[Lemma 4.2]{FuGuzman},
we see that $z \cdot t$ is single-valued for all $e \in \Delta_1(T)$. 
\end{proof}

\begin{lemma}\label{lem:vwg713}
For any $v \in \mathring{\VV}_r^2(\Twf)$, with $r \ge 1$, and any integer $s \geq 0$, there exists $w \in \LL_r^1(\Twf)$ and $g \in V_{r-1}^2(\Twf)$ such that
\begin{align}\label{vwg713}
\mu^s v &= \curl (\mu^{s+1} w) + \mu^{s+1} g.
\end{align}
\end{lemma}
\begin{proof}
From Lemma \ref{curlProp},
there exists $z \in [\pol_r(\Twf)]^3$ and $\gamma \in [\pol_{r-1}(\Twf)]^3$ satisfying \eqref{eqn:curlidentity} with $z \times \grad \mu$  continuous on $F$ for each $F \in \Delta_2(T)$ and $z \cdot t$ is single-valued for all $e \in \Delta_1(T)$. Let $\{F_i\}_{i=0}^3$ be the four faces of $T$. For each $i$ we choose $b_i \in [\pol_r(F_i)]^2$ so that $b_i={z_{F_i}}$ on $\partial F_i$, which we are allowed to do since \MJN{$z \times \grad \mu$} \AL{is} continuous on $F_i$.  Since $z \cdot t$ is single\AL{-}valued for all $e \in \Delta_1(T)$, we have that \MJN{$b_i \cdot t|_e=b_j \cdot t|_e$} if $e =F_i \cap F_j$. Hence, using the curl-conforming N\'ed\'elec degrees of freedom of the second kind \cite{Nedelec86}, there exists $w_1 \in [\pol_r(T)]^3$ such that \AL{${(w_1)_{F_i}} =b_i \text{ on } F_i \text{ for } 0 \le i \le 3.$}

Since ${z_{F_i} - b_i }\in \mathring{\LL}_r^1(\Fct_i)$, according to Lemma \ref{lemmaaux861}, there exists $a_i \in \LL_r^1(\Twf)$ such that $\text{supp }(a_i) \subset K_i$ and ${(a_i)_{F_i}}  ={ z_{F_i} - b_i}$ on $F_i$. We set $w_2:=\sum_{i=0}^3 a_i$ and finally $w:=\frac{1}{s+1} (w_1+w_2) \in {\LL_r^1(\Twf)}$. Hence, 
\begin{alignat*}{1}
(s+1) { w_{F_i}} ={(w_1)_{F_i} +  ( w_2 )_{F_i}} = {b_i+ (a_i)_{F_i}} =z_{F_i} . 
\end{alignat*}
{From this we  deduce   $(s+1) \grad \mu \times w =  \grad \mu \times z$ on each $F_i$.}

Thus, \revj{there exists} $\phi \in [\pol_{r-1}(\Twf)]^3$ such that 
\begin{equation}
(s+1) \grad \mu \times w =  \grad \mu \times z+ \mu \phi =v +\mu ({\phi}-\gamma)\quad \text{ on } T.
\end{equation}

We write $\curl(\mu^{s+1} w)= (s+1)\mu^s \grad \mu \times w+ \mu^{s+1} \curl w=\mu^s v + \mu^{s+1} (\curl w -\gamma+\phi)$. Setting $g:=-(\curl w -\gamma+\phi)$, we have that \eqref{vwg713} holds.  Finally, since $\mu^s v \cdot n$ and $\curl(\mu^{s+1} w)\cdot n$   are single-valued on interior faces, $\mu^{s+1} g  \cdot n$ is 
single-valued. Because $\mu$ is continuous {and strictly positive in the interior of $T$}, this implies $g \cdot n$ is single-valued on interior faces, and thus $g \in V_{r-1}^2(\Twf)$. 
\end{proof}

\begin{lemma}\label{lem10}
Let $r \ge 0$  and $s \ge 0$ \MJN{be integers}. For any $g \in \mathring{V}_r^2(\Twf)$ there exists $\psi \in \LL_{r+1}^1(\Twf)$ and $\gamma \in \mathring{\VV}_r^2(\Twf)$ such that
\begin{align*}
	\mu^{s} g &= \curl (\mu^{s} \psi) + \mu^{s} \gamma.
\end{align*}
\end{lemma}
\begin{proof}
\revj{Let $\{F_i\}_{i=0}^3$ be the four faces of $T$ \MJN{so that} $g_{F_i} \in H({\rm div}_{F_i}; F_i)$ by Lemma \ref{lemmaVr0div}}.  We  use the (two-dimensional) divergence-conforming N\'ed\'elec degrees of freedom 
to construct $p_i \in  [\pol_r(F_i)]^2$ so that for $r \ge 1$,
\begin{equation*}
p_i \cdot (n_F \times t) = g_{F_i} \cdot  (n_F \times t)  \qquad \text{ on } e,  \forall e \in \Delta_1(F_i),
\end{equation*}
where $t$ is \AL{the unit vector} tangent to the edge $e$.  If $r=0$, we can satisfy the above equation for two of the three edges, however, on the third edge the equation will be \MJN{automatically satisfied} since $\divFi  (g_{F_i} -p_i)=0$. 

{Using  $g_{F_i} -p_i  \in \mathring{V}_{{\rm div}, r}^1(\Fct_i)$ we have that $\dive (g_{F_i}-p_i) \in \mathring{V}_{r-1}^2(\Fct_i)$}. By Corollary \ref{cor:rotdiv},
there exists $m_i \in \mathring{\LL}_r^1(\Fct_i)$ so that $\divFi m_i=\divFi(g_{F_i} -p_i)$ on $F_i$. Thus, if we let $\theta_i: =p_i+m_i $ we have   $\theta_i \in \LL_r^1(\Fct_i)$ and  $g_{F_i} -\theta_i \in  \mathring{V}_{{\rm div}, r}^1(\Fct_i)$ with $\divFi(g_{F_i}-\theta_i)=0$. By 
Corollary \ref{cor:rotdiv}, there exists $\kappa_i \in \mathring{\LL}_{r+1}^0(\Fct_i)$ such that $\text{rot}_{F_i} \kappa_i=g_{F_i} -\theta_i$.  Since $\kappa_i$ vanishes on $\partial F_i$ there exists  $ \beta_i \in \mathring{\LL}_{r+1}^0(\Twf)$ with $\text{supp }(\beta_i) \subset K_i$ such that $\beta_i=\kappa_i$ on $F_i$.  We let $\psi=\sum_{i=0}^3 \beta_i n_{F_i} \in \mathring{\LL}_{r+1}^1(\Twf)$. Note that this immediately implies that $\grad \mu \times \psi \equiv 0$ on $T$. Also, we  have that 
\begin{alignat*}{1}
\curl \psi=\grad \beta_i \times n_{F_i}= {\rm rot}_{F_i}\, \kappa_i= g_{F_i}-\theta_i \qquad \text{ on } F_i. 
\end{alignat*}
Setting $\gamma=g-\curl \psi$ we see that  $\gamma \in \mathring{V}_r^2(\Twf)$ . Moreover,  noting, in addition, to the above equation,  that $\curl \psi|_{F_i}=(\curl \psi)_{F_i}$ since {$\curl \psi \cdot n_{F_i}=0$} on $F_i$, we see that $\gamma_{F_i}=\theta_i \in \LL_r^1(\Fct_i)$ and, hence, $\gamma \in \mathring{\VV}_r^2(\Twf)$. Finally, since $\grad \mu \times \psi \equiv 0$ we have  $\curl(\mu^{s} \psi)=\mu^{s} \curl \psi=\mu^{s} (g-\gamma)$. 
\end{proof}

\begin{lemma}\label{step2}
Let $r \ge 1, s \ge 0$ \MJN{be integers.  Then }for any $v \in \mathring{\VV}_r^2(\Twf)$ such that $\dive(\mu^s v)=0$ on $T$ there exists $w \in \LL_r^1(\Twf)$ and $g \in \mathring{\VV}_{r-1}^2(\Twf)$ satisfying $\mu^s v = \curl (\mu^{s+1} w) + \mu^{s+1} g$.
\end{lemma}
\begin{proof}
By \MJN{Lemma \ref{lem:vwg713}, there exists} $w_1 \in \LL_r^1(\Twf)$ and ${g_1} \in V_{r-1}^2(\Twf)$  satisfying
\begin{align}
\mu^s v = \curl (\mu^{s+1} w_1) + \mu^{s+1} g_1.
\end{align}
By our hypothesis we have  $0=\dive(\mu^{s+1} g_1)= \mu^s( (s+1) \grad \mu \cdot g_1 + \mu \dive g_1)$. Hence, $ (s+1) \grad \mu \cdot g_1 + \mu \dive g_1=0$ on $T$ which implies $(\grad \mu) \cdot g_1 =0$ on $\partial T$. In other words, we have $g_1 \in   \mathring{V}_{r-1}^2(\Twf)$. We then apply Lemma \ref{lem10} to write   $\mu^{s+1} g_1= \revj{\curl}(\mu^{s+1} w_2)+ \mu^{s+1} g_2$ where  $w_2 \in \LL_r^1(\Twf)$ and  $g_2 \in   \mathring{\VV}_{r-1}^2(\Twf)$. The proof is complete if we set $w:=w_1+w_2$ and $g=g_2$. 
\end{proof}

\revjtwo{We need one final preliminary lemma. }
\begin{lemma}\label{lemmaextra}
\revjtwo{Let $s \ge 0$ and let $g \in  \mathring{\VV}_0^2(\Twf)$, then there exists $w \in [\pol_0(T)]^3$ such that
\begin{equation}
\mu^s g=\curl(\mu^{s+1} w).
\end{equation}
}
\end{lemma}
\begin{proof}
\revjtwo{It is easy to see that $g \in \mathring{V}_0^2(\Ta)$. Hence by  \cite[Lemma 4.3]{FuGuzman} there exists $w \in [\pol_0(T)]^3$ such that $\curl(\mu^{s+1} w)= \mu^s g$.} 
\end{proof}

Now we can prove parts (i) and (ii) of Theorem \ref{curlthm}. 
\begin{proof}[of part (i) of Theorem \ref{curlthm}]
 \revjtwo{If $r=0$ the result follows immediately from Lemma \ref{lemmaextra} with $s=0$. Now we consider the case  $r \ge 1$.} Let $0 \leq j \leq r-1$. Assume that we have found $w_{\AL{r-j}}, \dots, w_{\AL{r}}$ with $w_{\revjtwo{\ell}} \in \LL_{\revjtwo{\ell}}^1(\Twf)$ and $g_{r-(j+1)} \in \mathring{\VV}_{r-(j+1)}^2(\Twf)$ such that
\begin{align*}
	v &= \curl (\mu w_r + \mu^2 w_{r-1} + \cdots + \mu^{j+1} w_{r-j}) + \mu^{j+1} g_{r-(j+1)}.
\end{align*}
Since $\dive(\mu^{j+1} g_{r-(j+1)})=0$ on $T$,  \revjtwo{if we assume that $0 \le j <r-1$}, we apply apply Lemma \ref{step2} to get 
\begin{align*}
	\mu^{j+1}g_{r-(j+1)} &= \curl (\mu^{j+2}w_{r-(j+1)}) + \mu^{j+2}g_{r-(j+2)},
\end{align*}
where $w_{r-(j+1)} \in \AL{\LL}_{r-(j+1)}^1(\Twf)$ and $g_{r-(j+2)} \in \mathring{\VV}_{r-(j+2)}^2(\Twf)$. It follows that
\begin{align*}
	v &= \curl (\mu w_r + \mu^2 w_{r-1} + \cdots + \mu^{j+1} w_{r-j} + \mu^{j+2}w_{r-(j+1)}) + \mu^{j+2}g_{r-(j+2)}.
\end{align*}
Continuing by induction, \revjtwo{after taking care of the base case $j=0$}, we have
\begin{align*}
	v &= \curl (\mu w_r + \mu^2 w_{r-1} + \cdots + \mu^r w_1 ) +\mu^{r} g_{0},  \qquad  \text{ with } g_0 \in \mathring{\VV}_0^2(\Twf).
\end{align*}
\revjtwo{By Lemma \ref{lemmaextra} there exists $w_0 \in [\pol_0(T)]^3$} such that $\curl(\mu^{r+1} w_0)= \mu^r g_0$. Setting $w:=\mu w_r + \mu^2 w_{r-1} + \cdots + \mu^{r+1} w_0$ completes the proof.  
\end{proof}

\begin{proof}[of part (ii) of Theorem \ref{curlthm}] 
Set $\phi = v - \RTPi_0 v$, where  $\RTPi_0 v$ is the lowest-order Raviart-Thomas projection of $v$ on $T$. Then  $\int_{F_i} \phi \cdot n_{F_i}\MJN{\,dA} = 0$ for each $F_i \in \Delta_2(T)$.
Applying Theorem \revj{\ref{2dseqs}},
there exists a $\rho_i \in \mathring{\LL}_{r+1}^1(\Fct_i)$ such that $\curl_{F_i} \rho_i = \phi \cdot n_{F_i}$ on ${F_i}$. By Lemma \ref{lemmaaux861} we can extend $\rho_i$ to a function $p_i \in \AL{\LL}_{r+1}^1(\Twf)$ with support only on $K_i$, such that $\revj{(p_i)_{F_i}} = \rho_i$ on $F_i$.  We let $p=\sum_{i=0}^3 p_i \in \LL_{r+1}^1(\Twf)$. Hence,  \revj{by \eqref{curlid}}, $\curl p \cdot n_{F_i} = \phi \cdot n_{F_i}$ on $F_i$.
Furthermore, there exists $s \in [\pol_1(T)]^3$ such that $\curl s = \RTPi_0 v$ \revj{where we used that $\dive \RTPi_0 v=0$ which follows by the commuting property of $\RTPi_0$ and the fact $\dive v=0$}. We set $\psi: = s+p  \in \LL_{r+1}^1(\Twf)$, then 
\begin{align*}
	v \cdot n_{F_i} = (\phi + \RTPi_0 v) \cdot n_{F_i} = (\curl \psi) \cdot n_{F_i} \qquad \text{ on } F_i. 
\end{align*}
Hence, we see that  $v - \curl \psi \in \mathring{V}_r^2(\Twf)$. By Lemma \ref{lem10} (\revj{with $s=0$}) we have $v - \curl \psi= \curl m+\gamma$ where  $m \in \LL_{r+1}^1(\Twf)$ and $\gamma \in \mathring{\VV}_r^2(\Twf)$. 
By part (i) of Theorem \ref{curlthm}, there exists $z \in \mathring{\LL}_{r+1}^1(\Twf)$ such that $\curl z = \gamma$. Setting  $w = \psi + m+z$  completes the proof. 
\end{proof}

\subsection{Surjectivity of the gradient operator on discrete local spaces} 
Finally, to show that the sequences in \eqref{WFbdryseqs}--\eqref{WFseqs} are exact, i.e.,
to complete the proof of Theorems \ref{bdryseqs}--\ref{seqs}, we establish the surjectivity
of the gradient operator onto spaces of curl--free functions.

\begin{theorem}\label{gradthm}
Let $r\ge 0$.  Then:
\begin{enumerate}[align=left, leftmargin=0pt, labelindent=\parindent]
\item[(i)] for any $v\in \mathring{V}_r^1(\Twf)$ (resp., $v\in V_r^1(\Twf)$) satisfying $\curl v=0$, there exists $w\in \mathring{\LL}_{r+1}^0(\Twf)$ 
(resp., $w\in \LL_{r+1}^0(\Twf))$, satisfying $\grad w = v$.
\item[(ii)] for any $v \in \mathring{\LL}_r^1(\Twf)$ (resp., $v \in \LL_r^1(\Twf)$) with $\curl v  = 0$, there exists a $w \in \mathring{S}_{r+1}^0(\Twf)$ (resp., $w \in S_{r+1}^0(\Twf)$) such that $\grad w = v$.
\item[(iii)] for any $v \in \mathring{S}_r^1(\Twf)$ (resp., $v \in S_r^1(\Twf)$) where $\curl v  = 0$, there exists a $w \in \mathring{S}_{r+1}^0(\Twf)$ (resp., $w \in S_{r+1}^0(\Twf)$) such that $\grad w = v$.
\end{enumerate}
\end{theorem}
%
 %
 \begin{proof}[proof of (i)]
 If $v\in \mathring{V}_r^1(\Twf)$ (resp., $v\in V_r^1(\Twf)$) is curl--free, then there exists $w\in \mathring{H}^1(T)$ \MJN{(resp., $w\in {H}^1(T)$)} such that
 $\grad w = v$.  Since $v$ is a piecewise polynomial of degree $r$ with respect to $\Twf$, it follows that $w$ is a piecewise polynomial 
 of degree $(r+1$), i.e., $w\in \mathring{\LL}_{r+1}^0(\Twf)$ 
(resp., $w\in \LL_{r+1}^0(\Twf))$.
\end{proof}
\begin{proof}[proof of (ii)]
Let $v \in \mathring{\AL{\LL}}_r^1(\Twf) \subset  \mathring{V}_r^1(\Twf)$ such that $\curl v = 0$. By part (i),  there exists $w \in \mathring{\LL}_{r+1}^0(\Twf)$ such that $\grad w = v$. However, clearly $w \in \mathring{S}_{r+1}^0(\Twf)$ \AL{since $\grad w \in \mathring{\LL}_r^1(\Twf)$}.  
\end{proof}
\begin{proof}[proof of (iii)] The proof is similar to (ii) and is omitted.\end{proof}
%
 
\section{Dimension Counts}\label{sec-dim}
Here, we give dimension counts for the spaces appearing in the local sequences \eqref{WFbdryseqs} and \eqref{WFseqs}.
As a first step, we state the dimensions of the N\'ed\'elec spaces $V_r^k(\Twf)$ and $\mathring{V}_r^k(\Twf)$,
and the Lagrange spaces $\LL_r^k(\Twf)$ and $\mathring{\LL}_r^k(\Twf)$ in Table \ref{tab:VLDim}.
These counts follow from well-known dimension formulas of these spaces and the fact that
$\Twf$ contains $9$ vertices, 1 internal vertex, $26$ edges, $8$ internal edges, $30$ faces, $18$ internal faces,
and $12$ tetrahedra.

The main step in the derivation of the dimension counts for the rest of the spaces in 
\eqref{WFbdryseqs}--\eqref{WFseqs} is to prove dimension counts for the subspaces
of the N\'ed\'elec spaces with additional smoothness on the faces of $T$, i.e., 
the dimension of $\VV_r^k(\Twf)$ and $\mathring{\VV}_r^k(\Twf)$.  The dimensions 
of the other spaces will then follow from the rank--nullity theorem.
%
%
%
%
%

\begin{table}
\caption{\label{tab:VLDim}Dimension counts of the canonical N\'ed\'elec  and Lagrange spaces with and without boundary conditions
with respect to the Worsey--Farin split.}
{\scriptsize 
\begin{tabular}{c|cccc}
& $k=0$ & $k=1$ & $k=2$ & $k=3$\\
\hline
$\dim V_r^k(\Twf)$ & $(2r+1)(r^2+r+1)$ & $2 (r + 1) (3 r^2 + 6 r + 4)$ & $3 (r + 1) (r + 2) (2 r + 3)$ & \MJN{$2 ( r+1) (r+2) (r+3)$}\\
$\dim \mathring{V}_r^k(\Twf)$ & $(2r-1)(r^2-r+1)$ & $ 2 (r + 1) (3 r^2 + 1)$ & \MJN{$3 (r+1) (r+2) (2r+1)$} & $2 r^3 + 12 r^2 + 22 r + 11$\\
$\dim \LL_r^k(\Twf)$ & $(2r+1)(r^2+r+1)$ & $3(2r+1)(r^2+r+1)$ & $3(2r+1)(r^2+r+1)$ & $(2r+1)(r^2+r+1)$\\
$\dim \mathring{\LL}_r^k(\Twf)$ & $(2r-1)(r^2-r+1)$ & $3 (2 r - 1) (r^2 - r + 1)$ & $ 3 (2 r - 1) (r^2 - r + 1)$ & $(r - 1) (2 r^2 - r + 2)$
\end{tabular}
}
\end{table}

\begin{table}\label{tab:SVDim}
\caption{Summary of dimension counts proved in Section \ref{sec-dim}. Here, the superscript $+$ indicates
the positive part of the number.}
{\scriptsize 
\begin{tabular}{c|cccc}
& $k=0$ & $k=1$ & $k=2$ & $k=3$\\
\hline
$\mathring{\VV}_r^k(\Twf)$ & --- & --- & $6r^3 + 21r^2 + 9r + 2$ & \MJN{$2r^3 +12r^2+10r+3$}\\
$S_r^k(\Twf)$ & $ 2r^3-6r^2+10r-2$ & $3r(2r^2-3r+5)$ & $6r^3+8r+2$ & $(2r+1)(r^2+r+1)$\\
$\mathring{S}_r^0(\Twf)$ & $\big(2(r-2)(r-3)(r-4)\big)^+$ & $\big(3(2r-3)(r-2)(r-3)\big)^+$ & $\big(2(r-2)(3r^2-6r+4)\big)^+$ & $(r-1)(2r^2-r+2)$
\end{tabular}
}
\end{table}

\begin{definition}\label{def:eF}
\revj{Let $T \in \mct$, then}  for each $F \in \Delta_2(T)$, let $e_F \in \Delta_{1}^I(\Fct)$ be an arbitrary, but fixed,  internal edge of $\Fct$. 
\end{definition}

{We  also define the ``jump" of a function across an edge.
\begin{definition}\label{defjump}
Consider the triangulation $\Fct$ of a face $F \in \Delta_2(T)$, and let the three triangles of $\Fct$ be labeled $Q_1, Q_2$, and $Q_3$. Let $e = \p Q_1\cap \p Q_2$ be an internal edge,  let $t$ be the unit vector tangent to $e$ pointing away from the split point $m_F$, and let $\sss = n_F\times t$ is unit vector orthogonal to both $t$ and $n_F$. 
Then, the jump of a function $p \in \pol_r(\Twf)$ \revj{across the edge $e$} is defined as
\begin{align*}
\jmp{p}_e &= (p|_{Q_1} - p|_{Q_2})\revj{\sss}.
\end{align*}
\end{definition}
}

\MJN{
\begin{remark}
In the remainder of the paper, the edge associated with vectors $t$ and $\sss$
should be inferred from their context.  For example, in the expression $\int_e v\cdot t\, ds$, 
the unit vector $t$ is understood to be the tangent vector of the edge $e$.
\end{remark}
}

\begin{lemma}\label{lemmaaux973}
Let $p \in V_{{\rm div}, r}^1(\Fct)$ and suppose that 
\begin{subequations}\label{aux973}
\begin{alignat}{2}
\int_{e_F} \jmp{p \cdot t}_{e_F} m\MJN{\,ds} &=  0 \quad &&\text{ for all }  m \in \pol_r(e_F)
\label{aux973-1} \\
\int_e  \jmp{p \cdot t}_e m\MJN{\,ds} & =0  \quad &&\text{ for all }  m \in \pol_{r-1}(e), \forall e \in \Delta_{1}^I(\Fct)\backslash \{e_F\}, \label{aux973-2}
\end{alignat}
\end{subequations}
where $t$ is the unit vector tangent to an edge $e$.  
Then, $p \in \LL_r^1(\Fct)$. 
\end{lemma}
\begin{proof}
Let $e \in \Delta_{1}^I(\Fct)$, and  \MJN{recall that $\sss$ is a unit vector} parallel to $F$ that is perpendicular to the edge $e$. Then since $p \in V_{{\rm div}, r}^1(\Fct)$, $\jmp{p \cdot \MJN{\sss}}=0$. In order to show that $p  \in \LL_r^1(\Fct)$ we need to show that  $\jmp{p \cdot t} =0$ for all internal edges  $e \in  \Delta_{1}^I(\Fct)$. By $\eqref{aux973-1}$ this is certainly true for $e=e_F$. In fact, this shows that $p$ is continuous \revj{across} $e_F$. Since \MJN{$\jmp{p \cdot s}=0$} on the two remaining edges this show that $p$ is continuous on the interior vertex $z$. In particular,  $\jmp{p \cdot t} (z)$ vanishes on the two remaining edges. Hence, using \eqref{aux973-2} shows that  $\jmp{p \cdot t} =0$.
\end{proof}

\begin{corollary}\label{cor782}
Let  $v \in \mathring{V}_r^2(\Twf)$ and suppose that for all $F \in \Delta_2(T)$, the following holds 
\begin{alignat*}{1}
\int_{e_F} \jmp{v_F \cdot t}_{e_F} m \MJN{\,ds}&=  0 \quad \text{ for }  m \in \pol_r(e_F),  \\
\int_e  \jmp{v_F \cdot t}_e m\MJN{\,ds} & =0  \quad \text{ for }  m \in \pol_{r-1}(e), \forall e \in \Delta_{1}^I(\Fct)\backslash \{e_F\}. 
\end{alignat*}
Then,  $v \in \mathring{\VV}_r^2(\Twf).$
\end{corollary}
\begin{proof}
The proof of Lemma \ref{lem10} shows that $v_F \in V_{\dive, r}^1(\Fct)$ for all  $F \in \Delta_2(T)$.
The result now follows by applying Lemma \ref{lemmaaux973}. 
\end{proof}
We see that the number of constraints in Corollary \ref{cor782} is  $4(3r+1)$. We  use this result to determine  the dimension of the space $\mathring{\VV}_r^2(\Twf)$.

\begin{lemma}\label{dimV2o}
Let $v \in \mathring{\VV}_r^2(\Twf)$ with $r \geq 1$. Then $v$ is fully determined by the following DOFs.
\begin{subequations}\label{VV2o_dofs}
\begin{align}
&v|_f \cdot n_f (a),  && \forall a \in \Delta_0(T),  && \forall f \in \Delta_2^I(\revj{\Ta}), a \subset \overline{f}, \label{VV2o:a}\\
&\int_{e} (v|_f \cdot n_f)  \kappa \, ds,  && \forall \kappa \in \pol_{r-2}(e), && \forall e \in \Delta_1(T), \forall f \in \Delta_2^{{I}}(\Twf), e \subset \overline{f}, \label{VV2o:bm}\\
&\int_{e} (v_F \cdot t)  \kappa \, ds,  && \forall \kappa \in \pol_{r-2}(e), && \forall e \in \Delta_1(\Fct)\backslash  \Delta_1^I(\Fct) ,  \forall F \in \Delta_2(T),\label{VV2o:bmm}\\
&\int_F v_F \cdot \kappa \, \MJN{dA}, && \forall \kappa \in \mathring{\LL}_r^1 (\Fct), && \forall F \in \Delta_2(T), \label{VV2o:c}\\
&\int_T v \cdot \kappa \, dx, && \forall \kappa \in V_{r-1}^2(\Twf). &&  \label{VV2o:d}
\end{align}
\end{subequations}
Here $t$ is tangent to $e$. Furthermore, $\dim \mathring{\VV}_r^2(\Twf) = 6r^3 + 21r^2 + 9r + 2.$
\end{lemma}
\begin{proof}
From Corollary \ref{cor782} we have 
\begin{align}
	\dim \mathring{\VV}_r^2(\Twf) \geq \dim \mathring{V}_r^2(\Twf) - 4(3r+1) = 6r^3 + 21r^2 + 9r + 2. \label{dimVV2o:lb}
\end{align}
We see that the number of DOFs from \eqref{VV2o:a} are $12=4\cdot 3$.  There are $6(r-1)$ DOFs for  \eqref{VV2o:bm} and $12(r-1)$  DOFs for  \eqref{VV2o:bmm}. We have $4 (3(r-1)(r-2)+\revj{6}(r-1)+2)$ DOFs from \eqref{VV2o:c}, and finally $3r(2r+1)(r+1) $ for  \eqref{VV2o:c}. Hence, the total number of DOFs \eqref{VV2o_dofs} is
\begin{align*}
	3r(2r+1)(r+1) + 12(r-1)(r-2) + 42(r-1) + 20 = 6r^3 + 21r^2 + 9r+2.
\end{align*}
Hence, we will prove that $\dim \mathring{\VV}_r^2(\Twf)=6r^3 + 21r^2 + 9r+2$ if we show  the constraints \eqref{VV2o_dofs} determine a function $v \in \mathring{\VV}_r^2(\Twf)$. To this end, suppose that the DOFs $\eqref{VV2o_dofs}$ vanish. The DOFs \eqref{VV2o:a} shows that $v$ vanishes  $ \forall a \in \Delta_0(T)$. The DOFs \eqref{VV2o:bm} and \eqref{VV2o:bmm} show that $v$ vanishes  $ \forall e \in \Delta_1(T)$. Also, the DOFs \eqref{VV2o:c} show that $v_F$ vanishes  $ \forall F \in \Delta_2(T)$. Thus, $v =0$ on $\partial T$ and so  $v=\mu w$ where $w \in V_{r-1}^2(\Twf)$. Finally, \eqref{VV2o:d} shows that $w$ vanishes. Thus, $v\equiv 0$.
\end{proof}

In a similar but significantly easier way we can show 
\begin{align}\label{eqn:LBonVV3}
\dim \VV_{r}^3(\Twf) \geq \dim V_{r}^3(\Twf) - 4(2(r+1)+r)= 2 (r^3 + 6 r^2 + 5 r + 2).
\end{align}
\begin{lemma}\label{V3dim}
\MJN{The space $ \VV_{r}^3(\Twf)$ has dimension $2 (r^3 + 6 r^2 + 5 r + 2)$,
 and therefore $\dim \mathring{\VV}_r^3(\Twf) = 2r^3 +12r^2+10r+3$.}
\end{lemma}

\begin{proof}
We can easily show that the following DOFs determine  $q \in  \VV_{r}^3(\Twf)$
\begin{subequations}\label{eqn:VVdofs}
\begin{align}
&\int_F q p \, dA, && \forall p \in \revj{\LL}_{r}^2(\Fct),\ \forall F \in \Delta_2(T), &\label{eqn:VV1} \\
&\int_T q p \, dx, && \forall p \in V_{r-1}^3(\Twf).  & \label{eqn:VV2}
\end{align}
\end{subequations}
 The number of DOFs are $2 (r^3 + 6 r^2 + 5 r + 2)$, which are exactly the number given by \eqref{eqn:LBonVV3}.
\end{proof}


\begin{theorem}\label{thm:Sodims}
The dimension counts in Table \ref{tab:SVDim} hold for $r\ge 1$.
\end{theorem}
\begin{proof}
Using the exactness of the sequences \eqref{WFbdryseqs} and the rank--nullity theorem, we have 
\begin{align*}
	\dim \mathring{S}_r^0(\Twf) - \dim \mathring{\LL}_{r-1}^1(\Twf) + \dim \mathring{\VV}_{r-2}^2(\Twf) - \dim \mathring{V}_{r-3}^3(\Twf) &= 0, \\
	\dim \mathring{S}_r^0(\Twf) - \dim \mathring{S}_{r-1}^1(\Twf) + \dim \mathring{\LL}_{r-2}^2(\Twf) - \dim \mathring{\VV}_{r-3}^3(\Twf) &= 0, \\
	\dim \mathring{S}_r^0(\Twf) - \dim \mathring{S}_{r-1}^1(\Twf) + \dim \mathring{S}_{r-2}^2(\Twf) - \dim \mathring{\LL}_{r-3}^3(\Twf) &= 0,\\
	\dim S_r^0(\Twf) - \dim \LL_{r-1}^1(\Twf) + \dim V_{r-2}^2(\Twf) - \dim V_{r-3}^3(\Twf)  &= 1, \\
\dim S_r^0(\Twf) - \dim S_{r-1}^1(\Twf) + \dim \LL_{r-2}^2(\Twf) - \dim V_{r-3}^3(\Twf) &= 1, \\
\dim S_r^0(\Twf) - \dim S_{r-1}^1(\Twf) + \dim S_{r-2}^2(\Twf) - \dim \LL_{r-3}^3(\Twf) &= 1.
\end{align*}
This along with Table \ref{tab:VLDim} and Lemmas \ref{dimV2o}--\ref{V3dim} give the result. 
\end{proof}

%

\begin{remark}
The dimension counts show that for small $r$, some of these spaces are trivialized. In particular, $S_r^0(\Twf) = \pol_r(T)$ for $r\in \{1,2\}$, and $S_1^1(\Twf) = [\pol_1(T)]^3$.
\end{remark}

\section{Degrees of Freedom and Commuting Projections}\label{sec-DOFs}

In this section, we provide unisolvent sets of degrees of freedom (DOFs)
for all spaces appearing in the exact (local) sequences \eqref{WFseqs}.
The DOFs are constructed such that they induce
commuting projections and in addition,
lead to global finite element spaces with suitable smoothness.

\begin{definition}\label{def:edgevecs}
\revj{Let $T \in \mct$ and  $F \in \Delta_2(T)$}. \MJN{Each edge $e \in \Delta_1^I(\FCT)$ is associated with two orthonormal vectors, $(t,\sss)$ (cf.~Definition \ref{defjump}).} 
Let $\revj{\rr}$ be the unit vector orthogonal to $t$ and $\revj{\sss}$ that is tangent to the interior face $f \in \Delta_2^I(\WFT)$ that contains \revj{the} edge $e$.
\end{definition}

\subsection{SLVV degrees of freedom}\label{sec:SLVVdofs}
We first give degrees of freedom (DOFs) for the local finite element space in the sequence \eqref{seq2}
which we refer to as the `SLVV' sequence due to the given notation.  These DOFs are constructed such 
that they induce projections that commute with the appropriate differential operators.

Now, we give degrees of freedom for $S_r^0(\WFT)$ for $r \geq 3$. When $r < 3$, this space reduces to $\pol_r(T)$.
\begin{lemma}\label{S0}
A function $q \in S_r^0(\WFT)$, with $r \geq 3$, is fully determined by the following degrees of freedom.
\vspace{-2em}
\begin{subequations}\label{S0dofs}
\begin{align}
&  && \span \text{No. of DOFs}  \nonumber \\
&q(a), &&   \forall a \in \Delta_0(T), & 4,  \label{S0:a} \\
& \grad q(a), &&    \forall a \in \Delta_0(T), & 12, \label{S0:b} \\
&\int_e q \kappa \, ds, &&   \forall \kappa \in \pol_{r-4}(e), \  \forall e \in \Delta_1(T), & 6(r-3), \label{S0:c} \\
&\int_e \frac{\partial q}{\partial n_e^{\pm}} \kappa \, ds, &&    \forall \kappa \in \pol_{r-3}(e), \  \forall e \in \Delta_1(T), & 12(r-2), \label{S0:d} \\
&\int_F {\rm grad}_F\, q \cdot \kappa \, dA, &&  \forall \kappa \in {\rm grad}_F \mathring{S}_r^0(\FCT), \   \forall F \in \Delta_2(T), & 6(r-2)(r-3), \label{S0:e} \\
&\int_F (n_F \cdot \grad q) \kappa \, dA, &&  \forall \kappa \in \mathcal{R}_{r-1}^0(\FCT), \  \forall F \in \Delta_2(T), & 6(r-2)(r-3),\label{S0:f} \\
&\int_T \grad q \cdot \kappa \, dx, &&  \forall \kappa \in \grad \mathring{S}_r^0(\WFT), \span 2(r-2)(r-3)(r-4),   \label{S0:g}
\end{align}
\end{subequations}
where $\frac{\partial }{\partial n_e^{\pm}}$ represents two normal derivatives to edge $e$, so that $n_e^+, n_e^-$ and $t$ form \revj{an orthonormal basis} of $\mathbb{R}^3$.
Then the DOFs \eqref{S0dofs} define the projection $\Pi_r^0 : C^\infty(T) \rightarrow S_r^0(\WFT)$.
\end{lemma}
\begin{proof}
The dimension of $S_r^0(\WFT)$ is $2r^3 - 6r^2 + 10r -2$, which is equal to the sum of the number of the given DOFs.

Let $q \in S_r^0(\WFT)$ such that $q$ vanishes on the DOFs \eqref{S0dofs}. On each edge $e \in \Delta_1(T)$, $q|_e = 0$ by DOFs \eqref{S0:a} - \eqref{S0:c}. Furthermore, $\grad q|_e = 0$ by DOFs \eqref{S0:b} and \eqref{S0:d}. Then $q|_F \in \mathring{S}_r^0(\FCT)$ for each $F \in \Delta_2(T)$, and \eqref{S0:e} yields ${\rm grad}_F\, q|_F = 0$. Hence $q|_F$ is constant, and since $q|_{\partial F} = 0$, it follows that $q|_F = 0$ for each $F \in \Delta_2(T)$.

Write $q = \mu p$, where $p \in \LL_{r-1}^0(\WFT)$. Since $\mu$ is a positive linear polynomial on each $K \in \revj{\TA}$, and $q|_{K} \in S_r^0(\WFK)$, it follows that $p \in S_{r-1}^0(\WFK)$, hence $p|_F \in S_{r-1}^0(\FCT)$. We have  $\grad q = \mu \grad p + p \grad \mu$, hence on $F$, $n_F \cdot \grad q|_F = p (n_F \cdot \grad \mu)|_F$. Since $\grad q|_{\partial F} = 0$, it follows that $p|_{\partial F} = 0$. Therefore $p \in \mathcal{R}_{r-1}^0(\FCT)$, so $p|_F = 0$ by \eqref{S0:f}. Now $\grad q|_{\partial T} = 0$, hence $q \in \mathring{S}_r^0(\WFT)$, and by \eqref{S0:g}, we have $\grad q = 0$. Therefore $q = 0$, which is the desired result.
\end{proof}

\begin{remark}
In two dimensions, the work of \cite{Percell1979} provided nodal degrees of freedom for the space $S_r^0(\FCT)$ with $r \geq 3$.
\end{remark}


Next, we need the following vector-calculus identity.  Its proof is found in the appendix.
\begin{lemma}\label{curlv-sgrad}
Let $e$ be an internal edge of $\FCT$, and let $t$ and \MJN{$\sss$} be unit vectors tangent and orthogonal to $e$, respectively, as in \MJN{Definition \ref{defjump}}. Let $v \in \LL_r^1(\WFT)$ for some $r \geq 0$. If $v \times n_F = 0$ on $F$, then $\jmp{\curl v \cdot t}_e = \jmp{\grad(v \cdot n_F) \cdot \MJN{\sss}}_e$.
\end{lemma}

Now we are ready to give the degrees of freedom for \MJN{$\LL_{r-1}^1(\WFT)$}.
\begin{lemma}\label{L1}
A function $v \in \LL_{r-1}^1(\WFT)$, with $r \geq 3$, is fully determined by the following degrees of freedom.
\vspace{-2em}
\begin{subequations}\label{eqn:L1dofs}
\begin{align}
&& \span \span   \text{No. of DOFs}& \nonumber \\
&v(a), &&  & 12,& \label{L1:a} \\
&\int_e v \cdot \kappa \, ds, &&   \forall \kappa \in [\pol_{r-3}(e)]^3,  \enspace  \forall e \in \Delta_1(T), &   18(r-2),& \label{L1:b} \\
&\int_e \jmp{\curl v \cdot t}_e \kappa \, ds, && \forall \kappa \in \pol_{r-3}(e), \  \forall e \in {\Delta_1^I(\FCT) \backslash\{e_F\}},  \ & \nonumber \\
& && \forall F \in \Delta_2(T),   &  {8(r-2)}, \label{L1:c} \\
&{\int_{e_F} \jmp{\curl v \cdot t}_{e_F} \kappa \, ds,} && \forall \kappa \in \pol_{r-2}(e_F) ,  \ \forall F \in \Delta_2(T),   &  {4(r-1)},&  \label{L1:d} \\
&\int_F (v \cdot n_F) \kappa \, dA, &&   \forall \kappa \in \mathcal{R}_{r-1}^0(\FCT),     \forall F \in \Delta_2(T), \span 6(r-2)(r-3),& \label{L1:e} \\
&\int_F {\rm curl}_F\, v_F \kappa \, dA, &&    \forall \kappa \in \mathring{V}_{r-2}^2(\FCT),     \forall F \in \Delta_2(T),  \span  6r^2 - 6r - 4,& \label{L1:f} \\
%
&\int_F v_F \cdot \kappa \, dA, &&   \forall \kappa \in {\rm grad}_F \mathring{S}_r^0(\Fct),     \forall F \in \Delta_2(T), \span 6(r-2)(r-3),& \label{L1:g} \\
&\int_T \curl v \cdot \kappa \, dx, &&   \forall \kappa \in \curl \mathring{\LL}_{r-1}^1(\WFT),    \span 4r^3 - 9r^2 - 7r + 21,& \label{L1:h} \\
&\int_T v \cdot \kappa \, dx, &&   \forall \kappa \in \grad \mathring{S}_{r}^0(\WFT), \span   2(r-2)(r-3)(r-4).& \label{L1:i}
\end{align}
\end{subequations}
Then the DOFs \eqref{eqn:L1dofs} define the projection $\Pi_{r-1}^1 : [C^\infty(T)]^3 \rightarrow \LL_{r-1}^1(\WFT)$.
\end{lemma}
\begin{proof}
The dimension of $\LL_{r-1}^1(\WFT)$ is $6r^3 - 9r^2 + 9r -3$, which is equal to the number of DOFs in \eqref{eqn:L1dofs}. Let $v \in \LL_{r-1}^1(\WFT)$ such that $v$ vanishes on the DOFs \eqref{eqn:L1dofs}. Then $v|_e = 0$ for each edge $e \in \Delta_1(T)$ by \eqref{L1:a}--\eqref{L1:b}, so $v_F \in \mathring{\LL}_{r-1}^1(\FCT)$ on each $F \in \Delta_2(T)$. From \MJN{\eqref{2dbdryseq2}}, we can see that ${\rm curl}_F\, v_F \in \mathring{V}_{r-2}^2(\FCT)$. Then \eqref{L1:f} yields ${\rm curl}_F\, v_F = 0$, and by the exactness of the sequence \MJN{\eqref{2dbdryseq2}} and \eqref{L1:g}, we have $v_F=0$. 

{Since $\curl v \cdot n_F = {\rm curl}_F\,v_F=0$ on $F$ it follows from Corollary \ref{cor782} and DOFs \eqref{L1:c}--\eqref{L1:d} that $\jmp{\curl v \cdot t}_e = 0$ for each $e \in \Delta_1^I(\FCT)$. Hence, by Lemma \ref{curlv-sgrad}, $v \cdot n_F|_F \in S_{r-1}^0(\FCT)$, and since $v \cdot n_F|_{\partial F} = 0$, we have $v \cdot n_F|_F \in \mathcal{R}_{r-1}^0(\FCT)$. Then $v \cdot n_F|_F = 0$ by \eqref{L1:e}. We therefore conclude  $v|_{\partial T} = 0$.}

Now $v \in \mathring{\LL}_{r-1}^1(\WFT)$, so $\curl v = 0$ by \eqref{L1:h}. Using the exactness of sequence \eqref{WFbdryseq2}, there exists a $p \in \mathring{S}_r^0(\WFT)$ such that $\grad p = v$. So by \eqref{L1:i}, $v = 0$, which is the desired result.
\end{proof}

Next, we can write the degrees of freedom for $V_{r-2}^2(\WFT)$.
\begin{lemma}\label{V2}
A function $w \in V_{r-2}^2(\WFT)$, with $r \geq 3$, is fully determined by the following degrees of freedom.
\begin{subequations}\label{V2dofs}
\begin{align}
& && \span \text{No. of DOFs} \nonumber \\
&\int_e \jmp{w \cdot t}_e q \, ds, &&  \forall q \in \pol_{r-3}(e), \  \forall e \in {\Delta_1^I(\FCT) \backslash\{e_F\}},  & \nonumber \\
& &&  \forall F \in \Delta_2(T),  &  {8(r-2)}, \label{V2:a} \\
&{\int_{e_F} \jmp{w \cdot t}_{e_F} q \, ds}, &&  \forall q \in \pol_{r-2}(e_F),  \ \forall F \in \Delta_2(T),  &  {4(r-1)}, \label{V2:b} \\
&\int_F w \cdot n_F q \, dA, &&  \forall q \in {V}_{r-2}^2(\FCT),  \ \forall F \in \Delta_2(T), \span  6r(r-1), \label{V2:c} \\
&\int_T (\dive w) q \, dx, &&  \forall q \in \mathring{V}_{r-3}^3(\WFT), \span  2r^3 - 6r^2 + 4r - 1, \label{V2:d} \\
&\int_T w \cdot q \, dx, &&  \forall q \in \curl \mathring{\LL}_{r-1}^1(\WFT), \span  4r^3-9r^2 - 7r+21. \label{V2:e}
\end{align}
\end{subequations}
Then the DOFs \eqref{V2dofs} define the projection $\Pi_{r-2}^2 : [C^\infty(T)]^3 \rightarrow V_{r-2}^2(\WFT)$.
\end{lemma}
\begin{proof}
The dimension of $V_{r-2}^2(\WFT)$ is $6r^3 - 9r^2 + 3r$, which is the number of DOFs in \eqref{V2dofs}. Let $w \in V_{r-2}^2(\WFT)$ such that $w$ vanishes on \eqref{V2dofs}. By DOF \eqref{V2:c}, we have $w \cdot n_F = 0$ on each $F \in \Delta_2(T)$. {By  DOFs \eqref{V2:a}--\eqref{V2:b}, and Corollary \ref{cor782} we have  $w \in \mathring{\mathcal{V}}_{r-2}^2(\WFT)$, so $\dive w = 0$ by \eqref{V2:d}
\MJN{and the exactness of \eqref{WFbdryseq2}}. \MJN{Using the exactness  of sequence \eqref{WFbdryseq2} again}, there exists a $v \in \mathring{\LL}_{r-1}^1(\WFT)$ such that $\curl v = w$. Therefore $w = 0$ by \eqref{V2:e}, which is the desired result.}
\end{proof}

Finally, we conclude this subsection with the DOFs of $V_{r-3}^3(\WFT)$.  The proof 
of the lemma is trivial, and so omitted.
\begin{lemma}\label{V3}
A function $p \in V_{r-3}^3(\WFT)$, with $r \geq 3$, is fully determined by the following degrees of freedom.
\vspace{-2em}
\begin{subequations}\label{V3dofs}
\begin{align}
& && \span \text{No. of DOFs} \nonumber \\
&\int_T p \, dx, &&  & 1, \label{V3:a} \\
&\int_T p q \, dx, &&  \forall q \in \mathring{V}_{r-3}^3(\WFT), &  2r(r-1)(r-2) - 1. \label{V3:b}
\end{align}
\end{subequations}
Then the DOFs \eqref{V3dofs} define the projection $\Pi_{r-3}^3 : C^\infty(T) \rightarrow V_{r-3}^3(\WFT)$.
\end{lemma}
%

\subsection{SLVV commuting diagram}
\begin{theorem}\label{thm:slvvgen}
Let $r \geq 3$. Given the definitions of the projections $\Pi_r^0, \Pi_{r-1}^1, \Pi_{r-2}^2,$ and $\Pi_{r-3}^3$ in Lemmas \ref{S0} - \ref{V3}, the 
following commuting properties are satisfied.
 \begin{subequations}\label{slvvcommgen1}
 \begin{align}
 \grad \Pi_r^0 q &= \Pi_{r-1}^1 \grad q, \quad \forall q \in C^\infty(T), \label{gradcommgen1}\\
 \curl \Pi_{r-1}^1 v &= \Pi_{r-2}^2 \curl v, \quad \forall v \in [C^\infty(T)]^3, \label{curlcommgen1}\\
 \dive \Pi_{r-2}^2 w &= \Pi_{r-3}^3 \dive w, \quad \forall w \in [C^\infty(T)]^3. \label{divcommgen1}
 \end{align}
 \end{subequations}
\end{theorem}
\begin{proof}

(i) \textit{Proof of \eqref{gradcommgen1}}. Given $q \in C^\infty(T)$, let 
$\rho = \grad \Pi_r^0 q - \Pi_{r-1}^1 \grad q \in \LL_{r-1}^1(\WFT)$. Then to show \eqref{gradcommgen1} holds, 
it is sufficient to show that $\rho$ vanishes on the DOFs \eqref{eqn:L1dofs} of Lemma \ref{L1}. 

Using \eqref{L1:a} and \eqref{S0:b}, we have $\rho(a) = \grad \Pi_r^0 q(a) - \Pi_{r-1}^1 \grad q(a) = 0$ for each $a \in \Delta_0(T)$. Using \eqref{L1:b} and \eqref{S0:d}, for each $e \in \Delta_1(T)$ and for any $\kappa \in [\pol_{r-3}(e)]^3$, 
\begin{align*}
\int_e \rho \cdot \kappa \, ds &= \int_e \grad (\Pi_r^0 q - q) \cdot \kappa \, ds  \\
&= \int_e \big(\frac{\partial}{\partial n_e^+}(\Pi_r^0 q - q)n_e^+ + \frac{\partial}{\partial n_e^-}(\Pi_r^0 q - q)n_e^- + \frac{\partial}{\partial t}(\Pi_r^0 q - q) t\big) \cdot {\kappa} \, ds \\
&= \int_e \frac{\partial}{\partial t}(\Pi_r^0 q - q) t \cdot \kappa \, ds=0,
\end{align*}
where the last line follows from \eqref{S0:a} {and \eqref{S0:c}}. Using \eqref{L1:c}, for each $e \in {\Delta_1^I(\FCT) \backslash\{e_F\}}$, with $F \in \Delta_2(T)$, and for any $\kappa \in \pol_{r-3}(e)$,
\begin{align*}
\int_e \jmp{\curl \rho \cdot t}_e \kappa \, ds &= \int_e \jmp{\curl \grad (\Pi_r^0 q - q) \cdot t}_e \kappa \, ds = 0,
\end{align*}
since the curl of the gradient is zero. By the same reasoning, {the DOFs  \eqref{L1:d} of $\rho$ vanish}. 
By \eqref{L1:e} and \eqref{S0:f}, for any $\kappa \in \mathcal{R}_{r-1}^0(\FCT)$,
\begin{align*}
\int_F (\rho \cdot n_F) \kappa \, dA &= \int_F (\grad(\Pi_r^0 q - q) \cdot n_F) \kappa \, dA = 0.
\end{align*}

Similarly, using \eqref{L1:f}, \revj{\eqref{curlid}, and the fact that $\curl \grad=0$}, $\int_F {\rm curl}_F\, \rho_F \, \kappa \, dA = 0$ for every $\kappa \in \mathring{V}_{r-2}^2(\FCT)$. Next, 
for $\kappa\in {\rm grad}_F\,\mathring{S}_r^0(\Fct)$,
\begin{align*}
\int_F \rho_F \cdot \kappa \, dA &= \int_F {\rm grad}_F\,(\Pi_r^0 q - q) \cdot \kappa \, dA = 0\revj{,}
\end{align*}
using \eqref{S0:e} and \eqref{L1:g}.

On the \MJN{macro-element $T$}, we use \eqref{L1:h} so that for all $\kappa \in \curl \mathring{\LL}_{r-1}^1(\WFT)$,
\begin{align*}
\int_T \curl \rho \cdot \kappa \, dx &= \int_T \curl \grad (\Pi_r^0 q - q) \cdot \kappa \, dx = 0.
\end{align*}
Finally, we use \eqref{L1:i} to see that for all $\kappa \in \grad \mathring{S}_r^0(\WFT)$,
\begin{align*}
\int_T \rho \cdot \kappa \, dx &= \int_T \grad(\Pi_r^0 q - q) \cdot \kappa \, dx = 0,
\end{align*}
by \eqref{S0:g}. Hence by Lemma \ref{L1}, $\rho = 0$, and the identity \eqref{gradcommgen1} is proved.\medskip

(ii) \textit{Proof of \eqref{curlcommgen1}}. Given $v \in [C^\infty(T)]^3$, let $\rho = \curl \Pi_{r-1}^1 v - \Pi_{r-2}^2 \curl v \in V_{r-2}^2(\WFT)$. To prove that \eqref{curlcommgen1} holds, we will show that $\rho$ vanishes on the DOFs \eqref{V2dofs} of Lemma \ref{V2}.

On the interior edges $e \in \Delta_1^I(\FCT){ \backslash\{e_F\}}$ of each face $F \in \Delta_2(T)$, and for all $q \in \pol_{r-3}(e)$, we have
\begin{align*}
\int_e \jmp{\rho \cdot t}_e \, {q} \, ds &= \int_e \jmp{\curl(\Pi_{r-1}^1 v - v)\cdot t}_e  \, {q}\, ds = 0,
\end{align*}
using \eqref{L1:c} and \eqref{V2:a}. {Similarly,  the DOFs \eqref{V2:c} of $\rho$ vanish}. 

To show that the DOFs \eqref{V2:c} of $\rho$ vanish we consider first constant functions and then functions orthogonal to constants. To this end, we use {\eqref{V2:c},  \eqref{L1:b}}, \revj{\eqref{curlid}} and Stokes Theorem, so that
\begin{align*}
\int_F \rho \cdot n_F \, dA &= \int_F {\rm curl}_F(\Pi_{r-1}^1 v - v)_{{F}} \, dA = 0,
\end{align*}
{{where} we used that $r \ge 3$}. Moreover, for any $p \in \mathring{V}_{r-2}^2(\FCT)$, from \eqref{V2:c} and \eqref{L1:f}, we have
\begin{align*}
\int_F \rho \cdot n_F \, p \, dA &= \int_F \curl_F(\Pi_{r-1}^1 v -v)_F p \, dA = 0.
\end{align*}
 It follows from \eqref{V2:d} that for all ${p \in \mathring{V}_{r-3}^3(\Twf)}$
\begin{align*}
\int_T (\dive \rho) p \, dx &= \int_T \dive \curl (\Pi_{r-1}^1 v -v) p \, dx = 0.
\end{align*}
Finally, for all $p \in \curl \mathring{\AL{\LL}}_{r-1}^1(\WFT)$, it follows from \eqref{L1:h} and \eqref{V2:e} \MJN{that}
\begin{align*}
\int_T \rho \cdot p \, dx &= \int_T \curl(\Pi_{r-1}^1 v - v) \cdot p \, dx = 0.
\end{align*}
Hence by Lemma \ref{V2}, $\rho = 0$, and the identity \eqref{curlcommgen1} is proved.

(iii) \textit{Proof of \eqref{divcommgen1}}. Given $w \in [C^\infty(T)]^3$, let $\rho = \dive \Pi_{r-2}^2 w - \Pi_{r-3}^3 \dive w \in V_{r-3}^3(\WFT)$. We will show that $\rho$ vanishes on the DOFs \eqref{V3dofs}, so that $\rho = 0$. 

First, by \eqref{V2:c}, \eqref{V3:a}, and  Stokes Theorem, we have 
\begin{align*}
\int_T \rho \, dx = \int_T \dive(\Pi_{r-2}^2 w - w) \, dx = \int_{\partial T} (\Pi_{r-2}^2 w - w) \cdot n \, \MJN{dA} = 0.
\end{align*}
Next, using \eqref{V2:d} and \eqref{V3:b}, for any $q \in \mathring{V}_{r-3}^3(\WFT)$,
\begin{align*}
\int_T \rho q \, dx &= \int_T \dive(\Pi_{r-2}^2 w - w) q \, dx = 0,
\end{align*}
since $\mathring{V}_{r-3}^3(\WFT) = \dive \mathring{\VV}_{r-2}^2(\WFT)$ \MJN{(cf.~Theorem \ref{bdryseqs})}.
Then by Lemma \ref{V3}, $\rho = 0$, and the identity \eqref{divcommgen1} is proved.

\end{proof}

\subsection{SSLV degrees of freedom}\label{sec:SSLVdofs}
{In this section we construct degrees of freedom for spaces in the sequence that takes the Lagrange finite element in the third slot,
i.e., sequence \eqref{seq3}. The third and last space are well--suited for fluid flow problems as we discuss in the introduction}.

%
We will define degrees of freedom for each of the spaces $S_{r-1}^1(\WFT), \LL_{r-2}^2(\WFT),$ and $V_{r-3}^3(\WFT)$ that induce  projections $\pi_{r-1}^1, \pi_{r-2}^2,$ and $\pi_{r-3}^3$, respectively, such that they satisfy commuting properties.  First, we provide
a unisolvent set of DOFs for $S_{r-1}^1(\WFT)$.
%
\begin{lemma}\label{S1}
A function $v \in S_{r-1}^1(\WFT)$, with $r \geq 2$, is fully determined by the following DOFs.
\begin{subequations}\label{S1dofs}
\begin{align}
& && & \text{No. of DOFs} \nonumber \\
&v(a), &&  \forall a \in \Delta_0(T), &  12, \label{S1:a} \\
&\curl v(a), &&  \forall a \in \Delta_0(T), &  12, \label{S1:b} \\
&\int_e v \cdot q \, ds, &&  \forall q \in [\pol_{r-3}(e)]^3, \  \forall e \in \Delta_1(T), &  18(r-2), \label{S1:c} \\
&\int_e \curl v \cdot q \, ds, &&  \forall q \in [\pol_{r-4}(e)]^3, \ \forall e \in \Delta_1(T), &  18(r-3), \label{S1:d} \\
&\int_F \curl_F v_F \, q \, dA, &&  \forall q \in \LL_{r-3}^0(\FCT) \revj{\cap L_0^2(F)}, \  \forall F \in \Delta_2(\revj{T}), \,\, \span \text{\small$6 r^2 - 30 r + \revj{36}$,} \label{S1:e} \\
&\int_F (v \cdot n_F) q \, dA, &&  \text{\small$\forall q \in \mathcal{R}_{r-1}^0(\FCT)$,} \ \text{\small$\forall F \in \Delta_2(\revj{T})$,} \span \text{\small$6(r-2)(r-3)$,} \label{S1:f} \\
&\int_F v_F \cdot q \, dA, && \text{\small$\forall q \in {\rm grad}_F \mathring{S}_r^0(\FCT)$} \ \text{\small$\forall F \in \Delta_2(\revj{T})$,} \span \text{\small$6 r^2 - 30 r + \revj{36}$} \label{S1:g} \\
%
%
&\int_F (\curl v)_F \cdot q \, dA, &&  \forall q \in \mathcal{R}_{r-2}^1(\FCT), \ \forall F \in \Delta_2(\revj{T}), &  12(r-3)^2, \label{S1:h} \\
&\int_T \curl v \cdot q \, dx, &&  \forall q \in \curl \mathring{S}_{r-1}^1(\WFT), \span  \text{\small$(4r-11)(r-3)(r-4)$,}  \label{S1:i} \\
&\int_T v \cdot q \, dx, && \forall q \in \grad \mathring{S}_r^0(\WFT), \span  \text{\small$2(r-2)(r-3)(r-4)$}. \label{S1:j}
\end{align}
\end{subequations}
Then the DOFs \eqref{S1dofs} define the projection $\pi_{r-1}^1 : [C^\infty(T)]^3 \rightarrow S_{r-1}^1(\WFT)$.
\end{lemma}
\begin{proof}
The dimension of $S_{r-1}^1(\WFT)$ is $6r^3 - 27r^2 + 51 r - 30$, which is equal to the number of DOFs in \eqref{S1dofs}.

Let $v \in S_{r-1}^1(\WFT)$ such that $v$ vanishes on \eqref{S1dofs}. Then DOFs \eqref{S1:a} and \eqref{S1:c} yield that $v|_e = 0$ for every $e \in \Delta_1(T)$. Furthermore, it follows from DOFs \eqref{S1:b} and \eqref{S1:d} that $\curl v|_e = 0$ for each $e \in \Delta_1(T)$. 

Since ${\rm \curl}_F\, v_F \in \mathring{\LL}_{r-2}^0(\FCT)$, there exists a function $\beta \in \LL_{r-3}^0(\FCT)$ such that $\curl_F v_F = \lambda_F \beta$, where $\lambda_F$ is the continuous linear function on $F$ such that $\lambda_F(\revj{m_F}) = 1$ at the split point $m_F$ and $\lambda_F|_{\partial F} = 0$.  \revj{We also note that \eqref{S1:e} holds for all $\AL{q \in \LL_{r-3}^0(\FCT)}$ since $\int_F \curl_F v_F \, dA=0$, which follows from integration by parts}.  Thus\revj{,} we have $\curl_F v_F = 0$ \revj{by choosing $q=\beta$}. From the exactness of sequence \eqref{2dbdryseq2}, it follows that $v_F=0$ by \eqref{S1:g}.
Since $\curl v$ is continuous {and $v_F=0$, by Lemma \ref{curlv-sgrad} }we have that $\grad(v \cdot n_F)|_F$ is continuous. Therefore, $v \cdot n_F|_F \in \mathcal{R}_{r-1}^0(\FCT)$, so $v \cdot n_F|_F = 0$ by \eqref{S1:f}.

\revj{Since $\curl v  \in \LL_{r-2}^2(\Twf) \cap \mathring{V}_{r-2}^2(\Twf)$ and $\dive \curl v=0$, we can apply Lemma \ref{lemmaLr0div} }to deduce that $(\curl v)_F \in \mathcal{R}_{r-2}^1(\FCT)$, where we also used that $(\curl v)_F =0$ on $\partial F$. Hence, by \eqref{S1:h}, we have that $(\curl v)_F=0$. We already \MJN{have that $\curl v \cdot n_F |_F = 0$}, so $\curl v|_F = 0$ on each face $F \in \Delta_2(T)$.

On the macro-elements, we use \eqref{S1:i} to see that $\curl v = 0$. By the exactness of sequence \eqref{WFbdryseq2}, there exists a $p \in \mathring{S}_r^0(\WFT)$ such that $\grad p = v$. Hence by \eqref{S1:j}, $v = 0$, which is the desired result.
\end{proof}

%
\MJN{We state the DOFs of $\LL_{r-2}^2(\WFT)$
in the following lemma.}
\begin{lemma}\label{L2}
A function $w \in \LL_{r-2}^2(\WFT)$, with $r \geq 3$, is fully determined by the following DOFs.
\begin{subequations}\label{L2dofs}
\begin{align}
& &&  \span \text{No. of DOFs} \nonumber \\
&w(a), &&  \forall a \in \Delta_0(T), &  12, \label{L2:a} \\
&\int_e w \cdot q \, ds, &&  \forall q \in [\pol_{r-4}(e)]^3, \  \forall e \in \Delta_1(T), & 18(r-3), \label{L2:b} \\
&\int_F (w \cdot n_F) q \, dA,  &&  \forall q \in {\LL}_{r-3}^0(\FCT),\ \forall F\in \Delta_2(T) &   6(r-2)(r-3) + 4, \label{L2:c}\\ 
%
&\int_{e} \jmp{\dive  w}_{e} q \, ds, && \forall q \in \pol_{r-3}(e), \ e \in {\Delta_1^I(\FCT)\backslash\{e_F\}},\ \forall F\in \Delta_2(T), & 8(r-2),\label{L2:d}\\
&\int_{e_F} \jmp{\dive w}_{e_F} q \, ds, && \forall q \in \pol_{r-4}({e_F}), \ e_F \in \Delta_1^I(\FCT),\ \forall F\in \Delta_2(T), & 4(r-3),\label{L2:e}\\
%
&\int_F {w_F}  \cdot q \, dA, &&  \forall q \in \mathcal{R}_{r-2}^1(\FCT), \ \forall F \in \Delta_2(T), & 12(r-3)^2, \label{L2:f} \\
&\int_T \dive w \, q \, dx, && \forall q \in \dive \mathring{\LL}_{r-2}^2(\WFT), \span 2(r-3)(r-2)(r+2)+3, \label{L2:g} \\
&\int_T w \cdot q \, dx, &&  \forall q \in \curl \mathring{S}_{r-1}^1(\WFT), \span  (4r-11)(r-3)(r-4). \label{L2:h}
\end{align}
\end{subequations}
The the DOFs \eqref{L2dofs} define the projection $\pi_{r-2}^2 : [C^\infty(T)]^3 \rightarrow \AL{\LL}_{r-2}^2(\WFT)$.
\end{lemma}
\begin{proof}
The dimension of $\LL_{r-2}^2(\WFT)$ is $3(2r-3)(r^2 - 3r + 3)$, which is equal to the number of DOFs in \eqref{L2dofs}.

Let \MJN{$w \in \LL_{r-2}^2(\WFT)$} such that $w$ vanishes on the DOFs \eqref{L2dofs}. Using DOFs \eqref{L2:a} and \eqref{L2:b}, we have that $w|_e = 0$ for every $e \in \Delta_1(T)$, hence $w \cdot n_F|_F \in \mathring{\LL}_{r-2}^0(\FCT)$. Therefore $w\cdot n_F|_F=0$ by \eqref{L2:c}. Using DOFs \eqref{L2:d}--\eqref{L2:e} and Lemma \ref{divjump}, we have that $\dive w|_F \in \revj{\LL}_{r-3}^2(\FCT)$ for each $F \in \Delta_2(T)$.  \revj{Hence, using Lemma \ref{lemmaLr0div} we deduce that ${\rm div}_F\, w_F$ is continuous} which implies that $w_F \in \mathcal{R}_{r-2}^1(\FCT)$. By \eqref{L2:f}, it follows that $w_F = 0$ on $F$.

Now we have that $w \in \mathring{\LL}_{r-2}^2(\WFT)$, \revj{so by \eqref{L2:g}, $\dive w = 0$}. Using the exactness property of sequence \MJN{\eqref{WFbdryseq3}}, there exists a function $p \in \mathring{S}_{r-1}^1(\WFT)$ such that $\curl p = w$. Then by \eqref{L2:h}, $w = 0$, which is the desired result.
\end{proof}
%
\begin{lemma}\label{V3a}
A function $p \in V_{r-3}^3(\WFT)$, with $r \geq 3$, is fully determined by the following DOFs.
\begin{subequations}\label{V3adofs}
\begin{alignat}{5}
& && \span \text{No. of DOFs} \nonumber \\
&\int_{e} \jmp{p}_{e} q \, ds, && \quad \forall q \in \pol_{r-3}(e), \ e \in {\Delta_1^I(\FCT) \backslash\{e_F\}},\ \forall F\in \Delta_2(T), &\qquad 8(r-2),\label{V3a:a}\\
%
&\int_{e_F} \jmp{p}_{e_F} q \, ds, && \quad \forall q \in \pol_{r-4}({e_F}), \ {e_F} \in \Delta_1^I(\FCT), \ \forall F \in \Delta_2(T), &  \qquad 4(r-3), \label{V3a:b} \\
&\int_T p \, dx,  &&  &  1, \label{V3a:c} \\
&\int_T p q \, dx, && \forall q \in \mathring{\mathcal{V}}_{r-3}^3(\WFT),  \span  2r^3 - 6r^2 - 8r + 27. \label{V3a:d}
\end{alignat}
\end{subequations}
Then the DOFs \eqref{V3adofs} define the projection $\pi_{r-3}^3 : C^\infty(T) \rightarrow V_{r-3}^3(\WFT)$.
\end{lemma}
\begin{proof}
The dimension of $V_{r-3}^3(\WFT)$ is $2r(r-1)(r-2)$, which is equal to the number of DOFs in \eqref{V3adofs}. 

Let $p \in V_{r-3}^3(\WFT)$ such that $p$ vanishes on the DOFs \eqref{V3adofs}. Then by \eqref{V3a:a}--\eqref{V3a:b}
and Lemma \ref{divjump}, $\jmp{p}_e = 0$ for every $e \in \Delta_1^I(\FCT)$ for each $F \in \Delta_2(T)$. Combined with \eqref{V3a:c}, it follows that $p \in \mathring{\mathcal{V}}_{r-3}^3(\WFT)$. So by \eqref{V3a:d}, $p = 0$.
\end{proof}

\subsection{SSLV commuting diagram}
\begin{theorem}\label{thm:sslv}
Let $r \geq 3$, and let $\Pi_r^0 : C^\infty(T) \rightarrow S_r^0(\WFT)$ be the projection defined in Lemma \ref{S0}, let $\pi_{r-1}^1 : [C^\infty(T)]^3 \rightarrow S_{r-1}^1(\WFT)$ be the projection defined in Lemma \ref{S1}, let $\pi_{r-2}^2 : [C^\infty(T)]^3 \rightarrow \LL_{r-2}^2(\WFT)$ be the projection defined in Lemma \ref{L2}, and let $\pi_{r-3}^3 : C^\infty(T) \rightarrow V_{r-3}^3(\WFT)$ be the projection defined in Lemma \ref{V3a}. 
Then the following commuting properties are satisfied.
 \begin{subequations}
 \begin{align}
 \grad \Pi_r^0 q &= \pi_{r-1}^1 \grad q, \quad \forall q \in C^\infty(T), \label{gradcomm2:a}\\
 \curl \pi_{r-1}^1 v &= \pi_{r-2}^2 \curl v, \quad \forall v \in [C^\infty(T)]^3, \label{curlcomm2:b}\\
 \dive \pi_{r-2}^2w &= \pi_{r-3}^3 \dive w, \quad \forall w \in [C^\infty(T)]^3. \label{divcomm2:c}
 \end{align}
 \end{subequations}
\end{theorem}

\begin{proof}
(i) \textit{Proof of \eqref{gradcomm2:a}.} Let $q \in C^\infty(T)$, and set $\rho = \grad \Pi_r^0 q - \MJN{\pi_{r-1}^1} \grad q\in S_{r-1}^1(\WFT)$. 
We show that $\rho$ vanishes on the DOFs \eqref{S1dofs}.

For each $a \in \Delta_0(T)$, $\rho(a) = \grad \Pi_r^0q(a) - \MJN{\pi_{r-1}^1} \grad q(a) = 0$ by {\eqref{S0:b}} and \eqref{S1:a}. Then, using \eqref{S1:b}, $\curl \rho(a) = \curl(\grad (\Pi_r^0 q - q)) = 0$. By \eqref{S1:c}, we have, for all $p \in [\pol_{r-3}(e)]^3$ on each $e \in \Delta_1(T)$, 
\begin{align*}
\int_e \rho \cdot p \, ds &= \int_e \grad(\Pi_r^0 q - q) \cdot p \, ds \\
&= \int_e \left(\frac{\partial}{\partial n_e^+}(\Pi_r^0 q - q) n_e^+ + \frac{\partial}{\partial n_e^-}(\Pi_r^0 q - q) n_e^- + \frac{\partial}{\partial t}(\Pi_r^0 q - q)t\right) \cdot p \, ds  \\
&= \int_e \frac{\partial}{\partial t}(\Pi_r^0 q - q) (p \cdot t)\, ds &\hfill \text{by \eqref{S0:d}}\\
&= -\int_e(\Pi_r^0 q - q)\frac{\partial}{\partial t}(p \cdot t) \, ds =0&\hfill \text{by \eqref{S0:a} and \eqref{S0:c}.}
\end{align*}
Next, using \eqref{S1:d}, for all $p \in [\pol_{r-4}(e)]^3$,
\begin{align*}
\int_e \curl \rho \cdot p \, ds &= \int_e \curl \grad (\Pi_r^0 q - q) \cdot p \, ds =0.
\end{align*}

On the faces, from \eqref{S1:e}, we have for all $p \in {\LL}_{r-3}^0(\FCT) \revj{\cap L_0^2(F)}$,
\begin{align*}
\int_F \curl_F \rho_F \, p \, dA &= \int_F {\rm curl}_F\, {\rm grad}_F(\Pi_r^0 q - q) p \, dA = 0.
\end{align*}
Using \eqref{S0:f} and \eqref{S1:f}, for all $p \in \mathcal{R}_{r-1}^0(\FCT)$,
\begin{align*}
\int_F (\rho \cdot n_F) p \, dA &= \int_F (n_F \cdot \grad(\Pi_r^0 q - q))\,  p \, dA = 0.
\end{align*}
Next, using \eqref{S1:g} and {\eqref{S0:e}}, we have for all $p \in {\rm grad}_F\,\mathring{S}_{r}^0(\FCT)$,
\begin{align*}
\int_F \rho_F \cdot p \, dA &= \int_F ({\rm grad}_F(\Pi_r^0 q - q)_F) \cdot p \, dA = 0.
\end{align*}
Then we use {\eqref{S1:h}}, so that for all $p \in \mathcal{R}_{r-2}^1(\FCT)$,
\begin{align*}
\int_F (\curl \rho)_F \cdot p \, dA &= \int_F (\curl(\grad(\Pi_r^0 q -q)))_{F} \cdot p \, dA = 0.
\end{align*}

On the macro-elements, we use \eqref{S1:i} so that, for all $p \in \curl \mathring{S}_{r-1}^1(\WFT)$, 
\begin{align*}
\int_T \curl \rho \cdot p \, dx &= \int_T \curl \grad(\Pi_r^0 q - q) \cdot p \, dx = 0.
\end{align*}
Lastly, we use \eqref{S0:g} and \eqref{S1:j} to see that, for every $p \in \grad \mathring{S}_r^0(\WFT)$,
\begin{align*}
\int_T \rho \cdot p \, dx &= \int_T \grad(\Pi_r^0 q -q) \cdot p \, dx = 0.
\end{align*}
Therefore, by Lemma \ref{S1}, $\rho = 0$, and the identity \eqref{gradcomm2:a} is proved.

(ii) \textit{Proof of \eqref{curlcomm2:b}}. Let $v \in [C^\infty(T)]^3$, and set $\rho = \curl \pi_{r-1}^1 v - \pi_{r-2}^2 \curl v\in \LL_{r-2}^2(\WFT)$. We  show that $\rho$ vanishes on the DOFs \eqref{L2dofs}. 

By \eqref{S1:b} and \eqref{L2:a}, $\rho(a) = \curl \pi_{r-1}^1 v(a) - \pi_{r-2}^2\curl v(a) = 0$. By \eqref{S1:d} and \eqref{L2:b}, for all $p \in [\pol_{r-4}(e)]^3$ where $e \in \Delta_1(T)$,
\begin{align*}
\int_e \rho \cdot p \, ds &= \int_e \curl(\pi_{r-1}^1 v - v) \cdot p \, ds = 0.
\end{align*}
By \eqref{L2:c}, \revj{\eqref{S1:c}} and \eqref{S1:e}, for every $p \in {\LL}_{r-3}^0(\FCT)$,
\begin{align*}
\int_F (\rho \cdot n_F) p \, dA &= \int_F \curl_F((\pi_{r-1}^1 v)_F-v_F) p \, dA = 0,
\end{align*}
\revj{where used that $r \ge 3$.}

Using \eqref{L2:d}, for all $p \in \pol_{r-3}(e)$, $e \in  {\Delta_1^I(\FCT)\backslash\{e_F\}}$ and $F \in \Delta_2(T)$, we have
\begin{align*}
\int_{e} \jmp{\dive \rho}_{e} p \, ds &= \int_{e} \jmp{\dive \curl(\pi_{r-1}^1 v - v)}_{e} p \, ds = 0.
\end{align*}
Similarly, \eqref{L2:e} yields that $\int_{e_F} \jmp{\dive \rho}_{e_F} p \, ds = 0$ for  $p \in \pol_{r-2}(e_F)$.
Next, using \eqref{L2:f}, for any $p \in \mathcal{R}_{r-2}^1(\FCT)$, we have
\begin{align*}
\int_F \rho_F \cdot p \, dA &= \int_F {( \curl(\pi_{r-1}^1 v - v))_F }  \cdot p \, dA = 0
\end{align*}
by \eqref{S1:h}.

By \eqref{L2:g} and for any $p \in \dive \mathring{\LL}_{r-2}^2(\WFT)$,
\begin{align*}
\int_T \dive \rho \, p \, dx &= \int_T \dive\curl(\pi_{r-1}^1 v - v) p \, dx = 0.
\end{align*}
Finally, by \eqref{S1:i}, \eqref{L2:h}, and for any $p \in \curl \mathring{S}_{r-1}^1(\WFT)$,
\begin{align*}
\int_T \rho \cdot p \, dx &= \int_T \curl(\pi_{r-1}^1 v - v) \cdot p \, dx = 0.
\end{align*}
Therefore, $\rho = 0$ by Lemma \ref{L2}, which is the desired result.

(iii) \textit{Proof of \eqref{divcomm2:c}}. Let $w \in [C^\infty(T)]^3$, and set $\rho = \dive \pi_{r-2}^2 w - \pi_{r-3}^3 \dive w\in V_{r-3}^3(\WFT)$.We  show that $\rho$ vanishes on the DOFs \eqref{V3adofs}.

First, we see from \eqref{L2:d} and \eqref{V3a:a} that for any $p \in \pol_{r-3}(e)$, ${e\in \Delta_1^I(\FCT) \backslash\{e_F\}}$ and $F \in \Delta_2(T)$, we have
\begin{align*}
\int_{e} \jmp{\rho}_{e} p \, ds &= \int_{e} \jmp{\dive (\pi_{r-2}^2 w -w)}_{e} p \, ds = 0.
\end{align*}
{Similarly,  $\rho$ vanish on the DOFs \eqref{V3a:b}}.

We then use \eqref{V3a:c}, \eqref{L2:c}, and the Stokes Theorem to see that
\begin{align*}
\int_T \rho \, dx &= \int_T \dive(\pi_{r-2}^2 w -w) \, dx = \int_{\partial T} (\pi_{r-2}^2 w - w) \cdot n_F \, dA = 0,
\end{align*}
\revj{where again we used that $r \ge 3$.}

Lastly, for any $p \in \mathring{\mathcal{V}}_{r-3}^3(\WFT)$,
\begin{align*}
\int_T \rho \, p \, dx &= \int_T (\dive \pi_{r-2}^2 w - w) p \, dx = 0
\end{align*}
by \eqref{L2:g}, \eqref{V3a:d} and  the fact $\dive \mathring{\LL}_{r-2}^2(\WFT) = \mathring{\mathcal{V}}_{r-3}^3(\WFT)$  \revj{which follows} by the exactness of sequence \eqref{WFbdryseq3}.
Therefore $\rho = 0$ by Lemma \ref{V3a}, which is the desired result.
\end{proof}

\subsection{SSSL degrees of freedom}\label{sec:SSSLdofs}

In this section, we  define degrees of freedom for each of the spaces $S_{r-2}^2(\WFT)$ and $\LL_{r-3}^3(\WFT)$ that induce commuting projections corresponding to the local exact sequence \eqref{seq4}.
First the DOFs for $S_{r-2}^2(\WFT)$ are given below.

\begin{lemma}\label{S2}
A function $w \in S_{r-2}^2(\WFT)$, with $r \geq 3$, is fully determined by the following DOFs.
\begin{subequations}\label{S2dofs}
\begin{align}
& && \span \text{No. of DOFs} \nonumber \\
&w(a), && &  12, \label{S2:a} \\
&\dive w(a), && &  4, \label{S2:b} \\
&\int_e w \cdot q \, ds, &&  \forall q \in [\pol_{r-4}(e)]^3, \ \forall e \in \Delta_1(T), &  18(r-3), \label{S2:c} \\
&\int_e (\dive w) q \, ds,  &&  \forall q \in \pol_{r-5}(e), \  \forall e \in \Delta_1(T), &  6(r-4), \label{S2:d} \\
&\int_F (w \cdot n_F) q \, dA, &&  \forall q \in \LL_{r-3}^0(\FCT), \ \forall F \in \Delta_2(T), \span 6(r-2)(r-3)+4, \label{S2:e} \\
&\int_F w_F \cdot q \, dA, && \forall q \in \mathcal{R}_{r-2}^1(\FCT), \ \forall F \in \Delta_2(T), &  12(r-3)^2, \label{S2:f} \\
&\int_F (\dive w)  q \, dA, && \forall q \in {{\LL}_{r-4}^2(\FCT)}, \ \forall F \in \Delta_2(T), &  6(r-3)(r-4) +4, \label{S2:g} \\
&\int_T (\dive w) q \, dx,  &&  \forall q \in \mathring{\LL}_{r-3}^3(\WFT), \span  (r-4)(2r^2-13r+23), \label{S2:i} \\
&\int_T w \cdot q \, dx,  && \forall q \in \curl \mathring{S}_{r-1}^1(\WFT), \span  (4r-11)(r-3)(r-4). \label{S2:j}
\end{align}
\end{subequations}
Then the DOFs \ref{S2dofs} define the projection $\MJN{\varpi_{r-2}^2} : [C^\infty(T)]^3 \rightarrow S_{r-2}^2(\WFT)$.
\end{lemma}
\begin{proof}
The dimension of $S_{r-2}^2(\WFT)$ is $6r^3 - 36r^2 + 80r - 62$, which is equal to the number of DOFs in \eqref{S2dofs}.

Let $w \in S_{r-2}^2(\WFT)$ such that $w$ vanishes on the DOFs \eqref{S2dofs}. Then from \eqref{S2:a} and \eqref{S2:c}, $w|_e = 0$ for each $e \in \Delta_1(T)$, and $\dive w|_e = 0$ by \eqref{S2:b} and \eqref{S2:d}. 

On each $F \in \Delta_2(T)$, $w \cdot n_F|_F \in \mathring{\LL}_{r-2}^0(\FCT)$, hence $w\cdot n_F|_F=0$ by \eqref{S2:e}.
\revj{By Lemma \ref{lemmaLr0div},} we have that $\dive_Fw_F$ is continuous. Hence, 
 $w_F \in \mathcal{R}_{r-2}^1(\FCT)$. Therefore \eqref{S2:f} yields $w_F = 0$. Now we have $\dive w \in {\LL}_{r-3}^2(\FCT)$ {and $\dive w$ vanishes on $\partial F$}. So, $\dive w|_F = 0$ by \eqref{S2:g}, hence $w \in \mathring{S}_{r-2}^2(\WFT)$.

On the macro-element, we have  $\dive w = 0$ by \eqref{S2:i} \MJN{and the exactness
of \eqref{WFbdryseq4}.  Likewise, the exactness of \eqref{WFbdryseq4} yields the existence of a}
$v \in \mathring{S}_{r-1}^1(\WFT)$ such that $\curl v = w$. Hence by \eqref{S2:j}, $w = 0$, which is the desired result.
\end{proof}
%
%
\begin{lemma}\label{L3}
A function $p \in \revj{\LL}_{r-3}^3(\WFT)$, with $r \geq 3$, is fully determined by the following DOFs.
\begin{subequations}\label{L3dofs}
\begin{alignat}{5}
& && \span \text{No. of DOFs} \nonumber \\
&p(a),  && &  4,  \label{L3:a} \\
&\int_e p q \, ds,  && \quad \forall q \in \pol_{r-5}(e),  \forall e \in \Delta_1(T), &  6(r-4), \label{L3:b} \\
%
%
&\int_F pq \, dA, &&  \quad \forall q \in {{\LL}_{r-4}^2(\FCT)}, \ \forall F \in \Delta_2(T), & \qquad 6(r-3)(r-4)+4, \label{L3:c} \\
&\int_T p \, dx,  && & 1, \label{L3:d} \\
&\int_T p q \, dx, && \quad \forall q \in \mathring{\LL}_{r-3}^3(\WFT), \span  (r-4)(2r^2 - 13r+23). \label{L3:e}
\end{alignat}
\end{subequations}
Then the DOFs \eqref{L3dofs} define the projection $\MJN{\varpi_{r-3}^3} : C^\infty(T) \rightarrow \AL{\LL}_{r-3}^3(\WFT)$.
\end{lemma}
\begin{proof}
The dimension of $\AL{\LL}_{r-3}^3(\WFT)$ is $(2r-5)(r^2-5r+7)$, which matches the number of DOFs in \eqref{L3dofs}.

Let $p \in \LL_{r-3}^3(\WFT)$ such that $p$ vanishes on the DOFs \eqref{L3dofs}. Then by \eqref{L3:a} and \eqref{L3:b}, $p|_e = 0$ for every $e \in \Delta_1(T)$. For each $F \in \Delta_2(T)$, we have that $p|_F \in {\mathring{\LL}_{r-3}^2(\FCT)}$, so $p|_F = 0$ by \eqref{L3:c}. Then by \eqref{L3:d}, we have $p \in \mathring{\LL}_{r-3}^3(\WFT)$, and by \eqref{L3:e}, $p = 0$.
\end{proof}

\subsection{SSSL commuting diagram}

\begin{theorem}\label{thm:sssl}
Recall that $\Pi_r^0 : C^\infty(T) \rightarrow S_r^0(\WFT)$ is the projection defined in Lemma \ref{S0}, $\pi_{r-1}^1 : [C^\infty(T)]^3 \rightarrow S_{r-1}^1(\WFT)$ is the projection defined in Lemma \ref{S1}, $\varpi_{r-2}^2 : [C^\infty(T)]^3 \rightarrow \MJN{S_{r-2}^2(\WFT)}$ is the projection defined in Lemma \ref{S2}, and $\varpi_{r-3}^3 : C^\infty(T) \rightarrow \revj{\LL}_{r-3}^3(\WFT)$ is the projection defined in Lemma \ref{V3a}. 
There holds, \revj{for $r \ge 3$},
%
 \begin{subequations}
 \begin{align}
 \grad \Pi_r^0 q &= \pi_{r-1}^1 \grad q, \quad \forall q \in C^\infty(T), \label{gradcomm3:a}\\
 \curl \pi_{r-1}^1 v &= \varpi_{r-2}^2 \curl v, \quad \forall v \in [C^\infty(T)]^3, \label{curlcomm3:b}\\
 \dive \varpi_{r-2}^2w &= \varpi_{r-3}^3 \dive w, \quad \forall w \in [C^\infty(T)]^3. \label{divcomm3:c}
 \end{align}
 \end{subequations}
\end{theorem}

\begin{proof}
(i) \textit{Proof of \eqref{gradcomm3:a}.} The identity \eqref{gradcomm3:a} holds by Theorem \ref{thm:sslv}.

(ii) \textit{Proof of \eqref{curlcomm3:b}.} Let $v \in [C^\infty(T)]^3$, and set $\rho = \curl \pi_{r-1}^1 v - \varpi_{r-2}^2 \curl v$. Then $\rho \in S_{r-2}^2(\WFT)$, so we must show that $\rho$ vanishes on the DOFs \eqref{S2dofs}. By \eqref{S1:b} and \eqref{S2:a}, $\rho(a) = \curl \pi_{r-1}^1 v(a) - \curl v(a) = 0$, and by \eqref{S2:b}, $\dive \rho(a) = \dive \curl \pi_{r-1}^1 v(a) - \dive \curl v(a) = 0$ for each $a \in \Delta_0(T)$.

For all $e \in \Delta_1(T)$ and for all $p \in [\pol_{r-4}(e)]^3$,
\begin{align*}
\int_e \rho \cdot p \, ds &= \int_e \curl (\pi_{r-1}^1 v - v) \cdot p \, ds = 0
\end{align*}
by \eqref{S1:d} and \eqref{S2:c}. Using \eqref{S2:d}, for all $p \in \pol_{r-5}(e)$, we have
\begin{align*}
\int_e \dive \rho \, p \, ds &= \int_e \dive \curl (\pi_{r-1}^1 v - v) p \, ds = 0.
\end{align*}

On each face $F \in \Delta_2(T)$, for every $p \in {\LL}_{r-3}^0(\FCT)$, 
\begin{align*}
\int_F (\rho \cdot n_F) p \, dA &= \int_F \curl_F ((\pi_{r-1}^1 v)_F - v_F) p \, dA = 0,
\end{align*}
by \eqref{S1:e}, \revj{\eqref{S1:c}} and \eqref{S2:e}. \revj{Here we used that $r \ge 3$}. 

For each $p \in \mathcal{R}_{r-2}^1(\FCT)$, we have
\begin{align*}
\int_F \rho_F \cdot p \, dA &= \int_F \revj{\big(}\curl (\pi_{r-1}^1 v - v)\revj{\big)}_F \cdot p \, dA = 0,
\end{align*}
where we used \eqref{S1:h} and \eqref{S2:f}. Next, for every $p \in \LL_{r-{4}}^2(\FCT)$, \eqref{S2:g} yields
\begin{align*}
\int_F \dive \rho \, p \, dA &= \int_F \dive \curl (\pi_{r-1}^1 v - v) \, p \, dA = 0.
\end{align*}

\revj{For} each $p \in \mathring{\LL}_{r-3}^3(\WFT)$, we use \eqref{S2:i} so that
\begin{align*}
\int_T \dive \rho \, p \, dx &= \int_T \dive \curl (\pi_{r-1}^1 v - v) p \, dx = 0.
\end{align*}
Finally, for all $p \in \curl \mathring{S}_{r-1}^1(\WFT)$, 
\begin{align*}
\int_T \rho \cdot p \, dx &= \int_T \curl (\pi_{r-1}^1 v - v) \cdot p \, dx = 0,
\end{align*}
by \eqref{S1:i} and \eqref{S2:j}. Therefore, by Lemma \ref{S2}, $\rho = 0$, and the identity \eqref{curlcomm3:b} is proved.

(iii) \textit{Proof of \eqref{divcomm3:c}.} Let $w \in [C^\infty(T)]^3$, and set $\rho = \dive \varpi_{r-2}^2 w - \varpi_{r-3}^3 \dive w$. Then $\rho \in \LL_{r-3}^3(\WFT)$, and we show  $\rho$ vanishes on the DOFs \eqref{L3dofs}.

For all $a \in \Delta_0(T)$, $\rho(a) = \dive\varpi_{r-2}^2 w(a) - \varpi_{r-3}^3 \dive w(a) = 0$ by \eqref{S2:b} and \eqref{L3:a}. On each edge $e \in \Delta_1(T)$ and for all $p \in \pol_{r-5}(e)$,
\begin{align*}
\int_e \rho \, p \, ds &= \int_e \dive(\varpi_{r-2}^2 w - w) p \, ds = 0,
\end{align*}
by \eqref{S2:d} and \eqref{L3:b}.

On each face $F \in \Delta_2(T)$, using \eqref{L3:c}, we have, for all $p \in \LL_{r-{4}}^2(\FCT)$,
\begin{align*}
\int_F \rho \, p \, dA &= \int_F \dive(\varpi_{r-2}^2 w - w) \, p \, ds = 0,
\end{align*}
by \eqref{S2:g}.

Now we use \eqref{L3:d} and  Stokes Theorem to see that 
\begin{align*}
\int_T \rho \, dx &= \int_T \dive(\varpi_{r-2}^2 w - w) \, dx = \int_{\partial T} (\varpi_{r-2}^2 w - w) \cdot n \, dA = 0
\end{align*}
by \eqref{S2:e} \revj{since $r \ge 3$}. Then by \eqref{S2:i} and \eqref{L3:e}, for any $p \in \mathring{\revj{\LL}}_{r-3}^3(\WFT)$, 
\begin{align*}
\int_T \rho \, p \, dx &= \int_T \dive (\varpi_{r-2}^2 w - w) p \, dx = 0.
\end{align*}
Hence $\rho = 0$ by Lemma \ref{L3}, and the identity \eqref{divcomm3:c} is proved.
\end{proof}

\section{Global spaces and commuting diagrams}\label{sec-Global}
In this section, we discuss the global finite element spaces induced by the degrees of freedom of Subsections \ref{sec:SLVVdofs}, \ref{sec:SSLVdofs}, and \ref{sec:SSSLdofs}, thereby extending the results of Section \ref{sec:LocalSeq}. 

Recall $\mathcal{T}_h$ is a triangulation of the polyhedral domain $\Omega \subset \mathbb{R}^3$, and  let $\THWFT$ be the Worsey-Farin refinement of $\mathcal{T}_h$. 
One of the main features of the induced global spaces
is their intrinsic smoothness on Worsey--Farin splits \MJN{(cf.~\cite{FloaterHu20} for related results)}.  To describe this property
in detail, we require some definitions.

\begin{definition}\label{def:Eset}
We define the set $\mathcal{E}(\THWFT)$ as the collection of edges that are internal to a Clough-Tocher split of a face $F \in \MJN{\Delta_2^{\revj{I}}(\mct)}$, \revj{i.e., $\mathcal{E}(\THWFT)=\{ e \in \Delta_1^I(\Fct): \forall F \in \Delta_2^I(\mct) \}$}. 
\end{definition}

\revj{We will use the following notation in this section}. Let $T_1$ and $T_2$ be adjacent tetrahedra in $\mathcal{T}_h$ that share a face $F$. Let $K_1$ and $K_2$ be tetrahedra in $\TA_1$ and $\TA_2$, respectively, such that $K_1$ and $K_2$ share the face $F$. Let $\FCT$ represent the triangulation of $\FCT$ in $\THWFT$, and let $\WFK_i$ be the triangulation of $K_i$ in $\THWFT$, where $1 \leq i \leq 2$. Given a simplex $S \in \Delta_s(\THWFT)$, with $0 \leq s \leq 3$, let $\chi(S)$ represent that characteristic function that equals $1$ on $S$ and $0$ otherwise. Without loss of generality, we choose $n_F = n_1,$ the outward normal to $T_1$ on $F$.

\begin{definition}\label{def:theta}
\revj{Using the above notation, let  $e \subset \Delta_1^I(\FCT)$. Furthemore, let $K_i^{1}, K_i^2 \in \Delta_3(\WFK_i)$, $1 \le i \le 2$  be such that  $e \subset \overline{K_i^j},  1\le i \le 2, 1\le j \le 2 $ and $K_1^2$ shares a face with $K_2^1$.}  \MJN{Then we define}
\begin{align*}
	{\theta_e(p) = |p|_{K_1^1} -p|_{K_1^2} +p|_{K_2^1} -p|_{K_2^2}| \qquad \text{ on } e.}
\end{align*}
\end{definition}
\begin{remark}\label{remark4}
{Note that if $\theta_e(p)=0$ if and only if $\jump{p_1}_e=\jump{p_2}_e$ where $p_i=p|_{T_i}$. }
\end{remark}

\begin{remark}\label{extensions}
The importance of the Worsey--Farin structure is that the natural extension of a piecewise polynomial from $\WFK_1$ to all of $\WFK_1 \cup \WFK_2$ maintains its original smoothness properties across the interior faces of $\WFK_2$, since all the faces of a given subtetrahedron in $\WFK_1$ are coplanar to the faces of the adjacent subtetrahedron of $\WFK_2$. {For example, if  $p \in \pol_r(\WFK_1) \cap C^1(K_1)$ then the natural extension of $p$, which we denote by $q$ satisfies $q|_{K_2}  \in \pol_r(\WFK_2) \cap C^1(K_2)$.} 
\end{remark}

We  show below that the projections defined in Sections \ref{sec:SLVVdofs}, \ref{sec:SSLVdofs}, and \ref{sec:SSSLdofs} induce the following global spaces.
\begin{align*}
	\mathcal{S}_r^0(\THWFT) &= \{ q \in C^1(\Omega) : q|_T \in S_r^0(\WFT) \, \forall T \in \mathcal{T}_h\}, \\
	\mathcal{S}_{r-1}^1(\THWFT) &= \{ v \in [C(\Omega)]^3 : \curl v \in [C(\Omega)]^3, \, v|_T \in S_{r-1}^1(\WFT) \, \forall T \in \mathcal{T}_h\}, \\
	\mathcal{S}_{r-2}^2(\THWFT) &= \{ w \in [C(\Omega)]^3 : \dive w \in C(\Omega), w|_T \in S_{r-2}^2(\WFT) \, \forall T \in \mathcal{T}_h\}, \\
	\mathcal{L}_{r-1}^1(\THWFT) &= \{ v \in [C(\Omega)]^3 : v|_T \in \LL_{r-1}^1(\WFT) \, \forall T \in \mathcal{T}_h\}, \\
	\mathcal{L}_{r-2}^2(\THWFT) &= \{ w \in [C(\Omega)]^3 : w|_T \in \LL_{r-2}^2(\WFT) \, \forall T \in \mathcal{T}_h\}, \\
	\mathscr{V}_{r-2}^2(\THWFT) &= \{w \in H(\dive ; \Omega) : w|_T \in V_{r-2}^2(\WFT) \, \forall T \in \mathcal{T}_h, \\
	& \qquad {\theta_e(w \cdot t) =0} \, \forall e \in \mathcal{E}(\THWFT)\}, \\
	\mathscr{V}_{r-3}^3(\THWFT) &= \{p \in L^2(\Omega) : p|_T \in V_{r-3}^3(\WFT) \, \forall T \in \mathcal{T}_h, \, \theta_e(p) = 0 \, \forall e \in \mathcal{E}(\THWFT)\}, \\
	\revj{\mathcal{L}_{r-3}^3(\THWFT)} &\revj{= \{ p \in C(\Omega):  p|_T \in \LL_{r-3}^3(\WFT) \, \forall T \in \mathcal{T}_h\},}\\
	V_{r-3}^3(\THWFT) &= \pol_{r-3}(\THWFT).
\end{align*}

\begin{remark}\label{rem:thetaprop}
Due to the singular edges formed through a Worsey-Farin refinement of a triangulation, the global space $\mathcal{L}_r^1(\THWFT)$ has the property
\begin{align}\label{eqn:thetacurl}
{\theta_e(\curl w \cdot t) = 0},
\end{align}
Moreover, any function $v \in \mathcal{L}_r^2(\THWFT)$ satisfies
\begin{align}\label{eqn:thetadiv}
\theta_e(\dive v) = 0.
\end{align}
We refer to \cite[Lemmas 5.7.3--5.7.4]{ALphd} for a proof of these results.
\end{remark}

Now we are ready to show that the global analogue of sequence \eqref{seq2} is induced by the local DOFs of Section \ref{sec:SLVVdofs}.
\begin{lemma}\label{globalS0}
The local degrees of freedom stated in Lemmas \ref{S0}, \ref{L1}, \ref{V2},
and \ref{V3} induce the global spaces $\mathcal{S}_r^0(\THWFT)$, $\mathcal{L}_{r-1}^1(\THWFT)$,
 $\mathscr{V}_{r-2}^2(\THWFT)$, and $V_{r-3}^3(\THWFT)$, respectively.

\end{lemma}
\begin{proof}\
\begin{enumerate}[align=left, leftmargin=0pt, labelindent=\parindent]
\item[(i)] Let $q_1 \in S_r^0(\WFT_1)$ and $q_2 \in S_r^0(\WFT_2)$ such that $q_1 - q_2$ vanishes on the DOFs \eqref{S0:a}--\eqref{S0:f} associated with the triangulation $\FCT$. We extend $q_1$ to $K_2$ according to Remark \ref{extensions}, and we set $p = q_1 - q_2$. Then by {the proof of } Lemma \ref{S0}, we have $p = 0$ and $\grad p = 0$ on $F$, therefore the function $q_1 \chi(K_1) + q_2 \chi(K_2)$ is $C^1$ across $F$. Therefore the DOFs \eqref{S0dofs} induce the global space $\mathcal{S}_r^0(\THWFT)$.
%
\item[(ii)] Let $v_1 \in \LL_{r-1}^1(\WFT_1)$ and $v_2 \in \LL_{r-1}^1(\WFT_2)$ such that $v_1 - v_2$ vanishes on the DOFs \eqref{L1:a}--\eqref{L1:g} associated with the triangulation $\FCT$ of the face $F$. We extend $v_1$ to $K_2$ as in Remark \ref{extensions}, and we set $w = v_1 - v_2$. Then by {by the proof of }  Lemma \ref{L1}, $w = 0$ on $F$. 
Therefore the local DOFs \eqref{eqn:L1dofs} induce the global space $\mathcal{L}_{r-1}^1(\THWFT)$.
%
\item[(iii)]Let $w_1 \in \revj{V}_{r-2}^2(\WFT_1)$ and $w_2 \in \revj{V}_{r-2}^2(\WFT_2)$ such that $w_1 - w_2$ vanishes on the DOFs \eqref{V2:a} - \eqref{V2:c} associated with the triangulation $\FCT$ of the face $F$. We extend $w_1$ to $K_2$ as in Remark \ref{extensions} and set $v = w_1 - w_2$. Then\revj{,} by Lemma \ref{V2}, $v \cdot n_F = 0$ on $F$\revj{,  which implies that $w_1 \chi(K_1) + w_2 \chi(K_2)$ is in $H(div)$ across $F$}. Furthermore, DOFs \eqref{V2:a} - \eqref{V2:b} imply that $\jmp{v \cdot t}_e = 0$ for each $e \in \Delta_1^I(\FCT)$, hence $\jmp{w_1 \cdot t}_e=\jmp{w_2 \cdot t}_e  $. \revj{By Remark \ref{remark4}, $\theta_e((w_1 \chi(K_1) + w_2 \chi(K_2)) \cdot t)=0$,} so the local DOFs \eqref{V2dofs} induce the global space $\mathscr{V}_{r-2}^2(\THWFT)$.

\item[(iv)]The DOFs \eqref{V3dofs} simply determine the piecewise polynomials $\pol_{r-3}(\WFT)$. Hence these DOFs naturally induce the global piecewise polynomial space $\pol_{r-3}(\THWFT)$.
\end{enumerate}
\end{proof}

Now we can see that the following sequence forms a complex by Theorem \ref{thm:slvvgen} for $r \geq 3$.
 \begin{alignat*}{5}
&\mathbb{R}\
{\xrightarrow{\hspace*{0.5cm}}}\
\mathcal{S}_{r}^0(\THWFT)\
&&\stackrel{\grad}{\xrightarrow{\hspace*{0.5cm}}}\
\mathcal{L}_{r-1}^1(\THWFT)\
&&\stackrel{\curl}{\xrightarrow{\hspace*{0.5cm}}}\
{\mathscr{V}}_{r-2}^2(\THWFT)\
&&\stackrel{\dive}{\xrightarrow{\hspace*{0.5cm}}}\
{V}_{r-3}^3(\THWFT)\
&&\xrightarrow{\hspace*{0.5cm}}\
 0.
\end{alignat*}
Furthermore, for $0 \leq k \leq 3$ and $r \geq 3$ we have commuting projections $\tilde{\Pi}_{r-k}^k$ such that $\tilde{\Pi}_{r-k}^k v|_T = \Pi_{r-k}^k(v|_T)$ for all $T \in \mathcal{T}_h$. Then by Theorem \ref{thm:slvvgen}, the following diagram commutes.
\begin{center}
\begin{tikzcd}
  \mathbb{R} \arrow[r] 
    & C^\infty(T) \arrow[d, "\tilde{\Pi}^{r}_3"] \arrow [r,"\grad"] & {[C^\infty(T)]}^3 \arrow[d,"\tilde{\Pi}^{1}_{r-1}"] \arrow[r,"\curl"] & {[C^\infty(T)]}^3 \arrow[r,"\dive"] \arrow[d,"\tilde{\Pi}^{2}_{r-2}"]& C^\infty(T) \arrow[r] \arrow[d,"\tilde{\Pi}^{3}_{r-3}"]& 0 \\
  \mathbb{R} \arrow[r] & \mathcal{S}_{r}^0(\THWFT) \arrow[r,"\grad"] & \mathcal{L}_{r-1}^1(\THWFT) \arrow[r,"\curl"] & \mathscr{V}_{r-2}^2(\THWFT) \arrow[r,"\dive"] & {V}_{r-3}^3(\THWFT) \arrow[r] & 0.
 \end{tikzcd}
 \end{center}
Next, we will show that the global analogue of sequence \eqref{seq3} is induced by the local DOFs of Section \ref{sec:SSLVdofs}.
\begin{lemma}\label{globalS1}
The local degrees of freedom stated in Lemmas \ref{S1}, \ref{L2}, and \ref{V3a} induce the global spaces 
$\mathcal{S}_{r-1}^1(\THWFT)$, $\mathcal{L}_{r-2}^2(\THWFT)$, and $\mathscr{V}_{r-3}^3(\THWFT)$, respectively.
\end{lemma}
\begin{proof}\
\begin{enumerate}[align=left, leftmargin=0pt, labelindent=\parindent]
\item[(i)]Let $v_1 \in S_{r-1}^1(\WFT_1)$ and $v_2 \in S_{r-1}^1(\WFT_2)$ such that $v_1 - v_2$ vanishes on the DOFs \eqref{S1:a} - \eqref{S1:h} associated with the triangulation $\FCT$. We extend $v_1$ to $K_2$ as in Remark \ref{extensions}, and we set $w = v_1 - v_2$. Then by {by the proof of }  Lemma \ref{S1}, $w = 0$ and $\curl w = 0$ on $F$, therefore the DOFs \eqref{S1dofs} induce the global space $\mathcal{S}_{r-1}^1(\THWFT)$.
\item[(ii)]
Let $w_1 \in \revj{\LL}_{r-2}^2(\WFT_1)$ and $w_2 \in \revj{\LL}_{r-2}^2(\WFT_2)$ such that $w_1 - w_2$ vanishes on the 
DOFs \eqref{L2:a} - \eqref{L2:f} associated with the triangulation $\FCT$ of the face $F$. We extend $w_1$ to $K_2$ as in Remark \ref{extensions}, and we set $v = w_1 - w_2$.  {Using the proof of Lemma \ref{L2} we can show that $v=0$ on $F$.  Hence,  $w_1 \chi(K_1) + w_2 \chi(K_2)$ is continuos accross $F$.}

%

\item[(iii)]Let $q_1 \in \revj{V}_{r-3}^3(\WFT_1)$ and $q_2 \in \revj{V}_{r-3}^3(\WFT_2)$ such that $q_1 - q_2$ vanishes on the DOFs \eqref{V3a:a} - \eqref{V3a:b} associated with the triangulation $\FCT$ of the face $F$. We can naturally extend \revj{$q_1$ to $K_2$} as in Remark \ref{extensions}. Let $p = q_1 - q_2$. As in Lemma \ref{V3a}, it follows from DOFs \eqref{V3a:a} - \eqref{V3a:b} that for each $e \in \Delta_1^I(\FCT)$, \revj{$\jmp{p}_e = 0$}, hence $\jmp{q_1}_e = \jmp{q_2}_e$, which implies that $\theta_e(q_1 \chi(K_1) + q_2 \chi(K_2))=0$. 
So, the local DOFs \eqref{V3adofs} induce the global space $\mathscr{V}_{r-3}^3(\THWFT)$.
\end{enumerate}
\end{proof}

Now we can see that the following sequence forms a complex by Theorem \ref{thm:sslv} for $r \geq 3$.
 \begin{alignat*}{5}
&\mathbb{R}\
{\xrightarrow{\hspace*{0.5cm}}}\
\mathcal{S}_{r}^0(\THWFT)\
&&\stackrel{\grad}{\xrightarrow{\hspace*{0.5cm}}}\
\mathcal{S}_{r-1}^1(\THWFT)\
&&\stackrel{\curl}{\xrightarrow{\hspace*{0.5cm}}}\
\mathcal{L}_{r-2}^2(\THWFT)\
&&\stackrel{\dive}{\xrightarrow{\hspace*{0.5cm}}}\
\mathscr{V}_{r-3}^3(\THWFT)\
&&\xrightarrow{\hspace*{0.5cm}}\
 0.
\end{alignat*}
 Furthermore, for $1 \leq k \leq 3$ and $r \geq 3$ we have commuting projections $\tilde{\pi}_{r-k}^k$ such that $\tilde{\pi}_{r-k}^k v|_T = \pi_{r-k}^k(v|_T)$ for all $T \in \mathcal{T}_h$, and by Theorem \ref{thm:sslv}, the following diagram commutes.
 \begin{center}
\begin{tikzcd}
  \mathbb{R} \arrow[r] 
    & C^\infty(T) \arrow[d, "\tilde{\Pi}^{r}_3"] \arrow [r,"\grad"] & {[C^\infty(T)]}^3 \arrow[d,"\tilde{\pi}^{1}_{r-1}"] \arrow[r,"\curl"] & {[C^\infty(T)]}^3 \arrow[r,"\dive"] \arrow[d,"\tilde{\pi}^{2}_{r-2}"]& C^\infty(T) \arrow[r] \arrow[d,"\tilde{\pi}^{3}_{r-3}"]& 0 \\
  \mathbb{R} \arrow[r] & \mathcal{S}_{r}^0(\THWFT) \arrow[r,"\grad"] & \mathcal{S}_{r-1}^1(\THWFT) \arrow[r,"\curl"] & \mathcal{L}_{r-2}^2(\THWFT) \arrow[r,"\dive"] & \mathcal{V}_{r-3}^3(\THWFT) \arrow[r] & 0.
 \end{tikzcd}
 \end{center}
Lastly, we will show that the global analogue of sequence \eqref{seq4} is induced by the local DOFs of Section \ref{sec:SSSLdofs}.
\begin{lemma}\label{globalS2}
The local degrees of freedom stated in Lemmas \ref{S2} and \ref{L3} induce the global spaces $\mathcal{S}_{r-2}^2(\THWFT)$ and $\mathcal{L}_{r-3}^{\revj{3}}(\revj{\THWFT})$, respectively.
\end{lemma}
\begin{proof}\
\begin{enumerate}[align=left, leftmargin=0pt, labelindent=\parindent]
\item[(i)]Let $w_1 \in {S}_{r-2}^2(\WFT_1)$ and $w_2 \in {S}_{r-2}^2(\WFT_2)$ such that $w_1 - w_2$ vanishes on the DOFs \eqref{S2:a} - \eqref{S2:g} associated with the triangulation $\FCT$. We extend $w_1$ to $K_2$ as in Remark \ref{extensions}, and we set $v = w_1 - w_2$. Then by Lemma \ref{S2}, $v = 0$ and $\dive v = 0$ on $F$. Therefore, the local DOFs \eqref{S2dofs} induce the global space $\mathcal{S}_{r-2}^2(\THWFT)$.
\item[(ii)]Let $q_1 \in \revj{\LL}_{r-3}^3(\WFT_1)$ and $q_2 \in \revj{\LL}_{r-3}^3(\WFT_2)$ such that $q_1 - q_2$ vanishes on the DOFs \eqref{L3:a} - \eqref{L3:d} associated with the triangulation $\FCT$ of the face $F$. We extend $q_1$ to $K_2$ as in Remark \ref{extensions}, and we set $p = q_1 - q_2$. It follows from Lemma \ref{L3} that $p = 0$ on $F$, which means $q_1 \chi(K_1) + q_2 \chi(K_2)$ is continuous across $F$. Therefore the local DOFs \eqref{L3dofs} induce the global space $\mathcal{L}_{r-3}^3(\THWFT)$.
\end{enumerate}
\end{proof}

Now we can see that the following sequence forms a complex by Theorem \ref{thm:sssl} for $r \geq 3$.
 \begin{alignat*}{5}
 &\mathbb{R}\
{\xrightarrow{\hspace*{0.5cm}}}\
{S}_{r}^0(\THWFT)\
&&\stackrel{\grad}{\xrightarrow{\hspace*{0.5cm}}}\
{S}_{r-1}^1(\THWFT)\
&&\stackrel{\curl}{\xrightarrow{\hspace*{0.5cm}}}\
{S}_{r-2}^2(\THWFT)\
&&\stackrel{\dive}{\xrightarrow{\hspace*{0.5cm}}}\
\revj{\mathcal{L}}_{r-3}^3(\THWFT)\
&&\xrightarrow{\hspace*{0.5cm}}\
 0.
\end{alignat*}
%
For $2 \leq k \leq 3$ and $r \geq 3$, we have commuting projections $\tilde{\varpi}_{r-k}^k$ such that $\tilde{\varpi}_{r-k}^k v|_T = \varpi_{r-k}^k(v|_T)$ for all $T \in \mathcal{T}_h$, and by Theorem \ref{thm:sssl}, the following diagram commutes.
 \begin{center}
\begin{tikzcd}
  \mathbb{R} \arrow[r] 
    & C^\infty(T) \arrow[d, "\tilde{\Pi}^{r}_3"] \arrow [r,"\grad"] & {[C^\infty(T)]}^3 \arrow[d,"\tilde{\pi}^{1}_{r-1}"] \arrow[r,"\curl"] & {[C^\infty(T)]}^3 \arrow[r,"\dive"] \arrow[d,"\tilde{\varpi}^{2}_{r-2}"]& C^\infty(T) \arrow[r] \arrow[d,"\tilde{\varpi}^{3}_{r-3}"]& 0 \\
  \mathbb{R} \arrow[r] & S_{r}^0(\THWFT) \arrow[r,"\grad"] & S_{r-1}^1(\THWFT) \arrow[r,"\curl"] & S_{r-2}^2(\THWFT) \arrow[r,"\dive"] & \revj{\mathcal{L}}_{r-3}^3(\THWFT) \arrow[r] & 0.
 \end{tikzcd}
 \end{center}
 


\bibliographystyle{abbrv}
\bibliography{refs}

\appendix

\section{Proof of Lemma \ref{curlv-sgrad}}
\begin{proof}
The set  $[t,s,n_F]^\top$ forms an orthonormal basis of $\mathbb{R}^3$, and therefore $v = a_t t + a_s s + a_n n_F$, with $a_t = v \cdot t$, $a_s = v \cdot s$, and $a_n = v \cdot n_F$. Since $v \times n_F = 0$ on $F$, we have $a_t = a_s = 0$ on $F$. Then, on $F$, 
\begin{align}\label{eqn:grada0}
{\rm grad}_F\, a_t = {\rm grad}_F\,a_s = 0.
\end{align}
We also have $\curl v = \grad a_t \times t + \grad a_s \times s + \grad a_n \times n_F$, and so
\begin{align}\label{curlvdott}
\curl v \cdot t &= (\grad a_s \times s + \grad a_n \times n_F) \cdot t.
\end{align}
Writing
$\grad a_s = (t \cdot \grad a_s) t + (s \cdot \grad a_s) s + (n_F \cdot \grad a_{\revj{s}}) n_F$, we find
\begin{align}\label{gradas}
(\grad a_s \times s) \cdot t = (n_F \cdot \grad a_s)(n_F \times s) \cdot t,
\end{align}
since $(t \times s) \cdot t = 0$ and $(s \times s) \cdot t = 0$.

Let $f$ be the interior face of $\WFT$ that contains $e$, and let $r$ be the unit vector tangent to $f$ and orthogonal to $t$. Then $r$ may be written $r = (r \cdot s) s + (r \cdot n_F) n_F$, therefore
\begin{align*}
n_F &= \frac{r - (r \cdot s) s}{r \cdot n_F}.
\end{align*}
Then by \eqref{eqn:grada0}, on $F$ we have
\begin{align}\label{ndotgradas}
n_F \cdot \grad a_s &= \frac{1}{r \cdot n_F} (r - (r \cdot s)s) \cdot \grad a_s
= \frac{1}{r \cdot n_F}(r \cdot \grad a_s).
\end{align}
Since $r$ is tangent to $f$ and $a_s$ is continuous, we have $\jmp{r \cdot \grad a_s}_e = 0$, which yields 
{$\jmp{n_F \cdot \grad a_s}_e = 0$ which in turn implies $\jmp{(\grad a_s \times s) \cdot t}_e = 0$ by \eqref{gradas}. It follows that  $\jmp{\curl v \cdot t}_e = \jmp{(\grad a_n \times n_F) \cdot t}_e$.}

We expand $\grad a_n$ in terms of $[t,s,n_F]^\top$ as
\begin{align*}
\grad a_n &= (t \cdot \grad a_n \revj{)} t + (s \cdot \grad a_n) s + (n_F \cdot \grad a_n) n_F.
\end{align*}
So $(\grad a_n \times n_F) \cdot t = (s \cdot \grad a_n (s \times n_F)) \cdot t$.
Because $(s \times n_F) \cdot t = 1$, it follows that $(\grad a_n \times n_F) \cdot t = s \cdot \grad a_n$. Therefore $\jmp{\curl v \cdot t}_e = \jmp{s \cdot \grad a_n }_e = \jmp{s \cdot \grad(v \cdot n_F)}_e,$ which is the desired result.
\end{proof}


\section{Miscellaneous Results}

\begin{lemma}\label{lemmaVr0div}
For any $g \in \mathring{V}_r^2(\Twf)$ we have that $g_F \in H({\rm div}_F; F)$ for $F \in \Delta_2(T)$. 
\end{lemma} 

\begin{proof}
\revj{ Let $e \in \Delta_1^I( \Fct)$, and let  $f$ be the corresponding an internal face of $\Twf$ that has $e$ as an edge. We let $t$ be a unit vector parallel to $e$ and set $\sss=t \times n_F$. Note that $\{n_F, \sss, t\}$ forms an orthonormal basis of $\bbR^3$. 
To prove $g_F \in H(\divF; F)$, it suffices to show $g_F \cdot \sss$ is single-valued on $e$. 

Let $n_f$ be a unit-normal to $f$. Since $n_f \cdot t=0$, we have that $n_f= (n_f \cdot \sss) \revj{\sss} + (n_f \cdot n_F) n_F$ and thus, $g \cdot n_f= g \cdot \sss (n_f  \cdot s)+ g \cdot n_F (n_f \cdot n_F)$ on $e$. However, $g \cdot n_F=0$ on $F$ by definition of $\mathring{V}_r^2(\Twf)$,
 and so  $g \cdot n_f= g \cdot \revj{\sss} (n_f  \cdot \revj{\sss})$ on $e$. Since $g \cdot n_f $ is single valued on $e$ (since $e \subset \partial f$ and $g \in  V_r^2(\Twf)$) we have that  $g \cdot \sss $ is single valued on  $e$. Finally, since $g_F \cdot \sss=g \cdot \sss$ we conclude    $g_F \in H({\rm div}_F; F)$.}
\end{proof}

\begin{lemma}\label{lemmaLr0div}
\revj{For any $g \in \LL_r^2(\Twf) \cap  \mathring{V}_r^2(\Twf)$ with $\dive g|_F$ continuous on $F$ we have that ${\rm div}_F g_F$ is continuous on $F$,  for $F \in \Delta_2(T)$. }
\end{lemma} 

\begin{proof}
\revj{Let  $K \in \Ta$ with $F \in \Delta_2(K) $. Since  $g \cdot n_F=0$ on $F$, we can write $g \cdot n_F{|_K}= \mu \psi$  {on $K$} for some $\psi \in \pol_{r-1}({\Kwf})$. However, since $g \cdot n_F$  is continuous on $K$  and $\mu$ is linear and positive on  $K$, it must be that $\psi$ is continuous on $K$. Since $n_F \cdot \grad (g \cdot n_F)= \psi \grad \mu \cdot n_F$ on $F$ this implies that $n_F \cdot \grad (g \cdot n_F)$ is continuous on $F$.
We can write ${\rm div}_F\, g_F= \dive g|_F-n_F \cdot \grad(g \cdot n_F)$ {on $F$} and, hence, ${\rm div}_F\, g_F$ is continuous on $F$.
}

\end{proof}

\MJN{\begin{lemma}\label{divjump}
Let $p \in V_r^3(\WFT)$ and $r \geq 0$. For $F \in \Delta_2(\WFT)$, if
\begin{subequations}
\begin{alignat}{2}
\int_{e} \jmp{p}_e q \, ds &= 0 \qquad \forall q \in \pol_{r}(e) \qquad &&e \in {\Delta_1^I(\FCT)\backslash\{e_F\}}, \quad \text{and} \label{eqn:ebar}\\
\int_{e_F} \jmp{p}_{e_F} \, ds &= 0 \qquad \forall q \in \pol_{r-1}(e_F) \qquad &&e_F \in \Delta_1^I(\FCT), \label{eqn:ehat}
\end{alignat}
\end{subequations}
then $p|_F$ is continuous.
\end{lemma}
\begin{proof}
We label the three triangles in $\Delta_2(\FCT)$ as $Q_1, Q_2$, and $Q_3$ such that $e_F = Q_1 \cap Q_2$. We let $p_i=p|_{Q_i}$ and let $z \in \Delta_0^I(\Fct)$. Since $p \in V_r^3(\WFT)$, condition \eqref{eqn:ebar} yields that $\jmp{p}_{e} = 0$ for both interior edges $e \in \Delta_1^I(\FCT)\backslash\{e_F\}$. It follows that $p_1(z) = p_2(z)$ and $p_2(z) = p_3(z)$, therefore $p$ is continuous at $z$. Hence $\jmp{p}_{e_F}(z) = 0$. Then, \eqref{eqn:ehat} shows that  $\jmp{p}_{e_F}=0$.
\end{proof}
}

\end{document}